\documentclass[11pt]{amsart}
\usepackage{amsfonts}
\usepackage{bbm}
\usepackage{amsfonts,amssymb,amsmath,amsthm}
\usepackage{url}
\usepackage{enumerate}
\usepackage{bbm}
\usepackage[all]{xy}
\usepackage[pdftex, colorlinks, citecolor=red, backref=page]{hyperref}
\usepackage{color,xcolor}

\newtheorem{theorem}{Theorem}[section]
\newtheorem{lemma}[theorem]{Lemma}
\newtheorem{definition}[theorem]{Definition}
\newtheorem{notion}[theorem]{}

\newtheorem{proposition}[theorem]{Proposition}
\newtheorem{corollary}[theorem]{Corollary}
\theoremstyle{definition}
\newtheorem{remark}[theorem]{Remark}

\newtheorem{example}[theorem]{Example}

\author{Qingnan An}
\address{School of Mathematics and Statistics, Northeast Normal University, Changchun, {\rm130024}, China}
\email{qingnanan1024@outlook.com}

\author{Chunguang Li}
\address{School of Mathematics and Statistics, Northeast Normal University, Changchun, {\rm130024}, China}
\email{licg864@nenu.edu.cn}

\author{Zhichao Liu}
\address{School of Mathematical Sciences,
Dalian University of Technology,
Dalian, {\rm116024}, China }
\email{lzc.12@outlook.com}

\keywords{ Classification; Real rank zero; Extension; Total K-theory}

\subjclass[2000]{Primary 46L35, Secondary 46L80 19K35}
\begin{document}

\title[A latticed total K-theory] {A latticed total K-theory}

\begin{abstract}
In this paper, a new invariant was built towards the classification of separable C*-algebras of real rank zero,  
which we call latticed total K-theory. 
A classification theorem is given in terms of such an invariant for a large class of separable C*-algebras of real rank zero arising from the  extensions
of finite and infinite C*-algebras. Many algebras with both finite and infinite projections can be classified.
\end{abstract}

\maketitle
\section{Introduction}

\subsection{Background}
In  the late 1980s, Elliott initiated the programme of classification of amenable ${\mathrm C}^*$-algebras (up to isomorphism) by their K-theoretical data \cite{Ell}. It has been shown that a large number of separable unital simple amenable ${\mathrm C}^*$-algebras can be classified by the standard Elliott invariant. 
For  purely infinite case, Kirchberg and Phillips classify all purely infinite simple separable amenable ${\mathrm C}^*$-algebras  which satisfy the UCT (see \cite{Rbook} for an overview); for stably finite case, combining the results of Elliott-Gong-Lin-Niu and Tikuisis-White-Winter, separable, unital, simple and infinite dimensional C*-algebras with finite nuclear dimension which satisfy the UCT are classified. \cite{EGLN,EGLN2,GLN1,GLN2,TWW}.

On the other hand, the classification of non-simple ${\mathrm C}^*$-algebras is far from being satisfactory. Gong \cite{G} first presented an example to show that the ordered graded $\mathrm{K}$-group is not sufficient
any more for non-simple AH algebras of real rank zero (and no dimension growth).
Elliott, Gong, and Su \cite{EGS} constructed such AD algebras by the reduction theorem of AH algebras.
Dadarlat and Loring \cite{DL3} also gave such an example for AD algebras directly. In 1997, Dadarlat and Gong \cite{DG} classified A$\mathcal{HD}$-algebras (see Definition \ref{AHD}) of real rank zero  by means of
scaled total $\mathrm{K}$-theory  together with a certain order structure (see also \cite{Ei}). Later, \cite{AELL, ALZ} push forward a classification for certain ASH algebras with one-dimensional local spectra (it includes much more general subhomogeneous building blocks). 
Recently, the first and third author construct two separable nuclear C*-algebras of stable rank one and real rank zero with the same ordered scaled total K-theory, but they are not isomorphic, a new invariant called total Cuntz semigroup $\underline{{\rm Cu}}$ built in \cite{AL jfa} can distinguish these two algebras, this leads to a further modification for Elliott Conjecture for real rank zero case \cite{AL}.

We point out that the total Cuntz semigroup $\underline{\mathrm{Cu}}$  
 does catch a lot of key information for algebras of stable rank one (\cite[Remark 5.16]{AL}), which may not have real rank zero.
If we only consider the algebras of real rank zero, 
there exists two infinite algebras with the same total Cuntz semigroup but they are not isomorphic. For example, $$\underline{\mathrm{Cu}}(\mathcal{O}_2)\cong \{0,\infty\}\cong \underline{\mathrm{Cu}}(\mathcal{O}_\infty\otimes \mathcal{Q}),$$ where $\mathcal{Q}$ is the UHF algebra whose ${\rm K}_0$-group is $(\mathbb{Q},\mathbb{Q}_+,1)$.  R\o rdam's example shows that there exists a simple infinite C*-algebra containing a finite projection \cite{R acta}, which is not purely infinite.
Until now, even for simple infinite C*-algebras, it is a nontrivial task to get a classification result. In order to classify certain infinite algebras and unify the finite and infinite case, we consider the semigroup of Murray-von Neumann equivalent classes of projections instead of the ${\rm K}_0$-group.

\subsection{New invariant}
In this present paper, we build a new invariant $\underline{V}_u$  towards the classification of separable C*-algebras of real rank zero, which we call the latticed total K-theory. It combines the semigroup of Murray-von Neumann equivalent classes of projections with the coefficients K-theory. This monoid has advantages both in the finite and infinite setting, rather than consider the $\mathrm{K}_0$-group, the semigroup of Murray-von Neumann equivalent classes contains all the information of projections and determines the $\mathrm{K}_0$-group. But under the Grothendieck construction, some key information will be lost (see \cite[Remark 7.8]{R acta} and \cite[2.3]{LS}). Note  that the projections in any C*-algebra $A$ of real rank zero satisfy Riesz decomposition property with respect to Murray-von Neumann equivalence, but this property can disappear in the  $\mathrm{K}_0$-group \cite{Z,Go}.

The class of C*-algebras of real rank zero we will classify are constructed from the extension theory.
This kind of works began from the extension of AF-algebras by Elliott \cite{Ell1}, and got a lot of attention from many
 works  \cite{L0,L1,L00,LR}. 
In \cite{R2}, R\o rdam  considered the six-term exact sequence of K-theory as an invariant and this was widely applied in the series of later classification works \cite{ERZ,ERR, GR}. The topic is closely linked to the extension problem, one can refer to \cite{EGKRT} for a review and many useful tables outlining the extension results.
Our results will also be based on those theories, but we will only focus on the invariant $\underline{V}(\cdot)$, since the latticed total K-theory $\underline{V}(\cdot)$
would catch the structure  of ideal lattice naturally. It is also worth mentioning that  the ordered six-term exact sequences of K-theory or even total K-theory doesn't distinguish the exact sequences
$$
0\to \mathcal{K}\oplus (\mathcal{O}_2\otimes \mathcal{K})
\to \widetilde{\mathcal{K}\oplus (\mathcal{O}_2\otimes \mathcal{K})}\to\mathbb{C}\to 0
$$
and
$$
0\to \mathcal{K}
\to \widetilde{\mathcal{K}}\to\mathbb{C}\to 0.$$
But the Murray-von Neumann semigroup performs well:
  $$\mathrm{V}(\widetilde{\mathcal{K}\oplus (\mathcal{O}_2\otimes \mathcal{K})})\ncong \mathrm{V}(\widetilde{\mathcal{K}}).$$ 
(For the case
of real rank greater than zero, we may refer the readers to the recent works \cite{LN, GaLN}.)


\subsection{Main results}
We will show that the total K-theory can be recovered from the latticed total K-theory; see Proposition \ref{recover prop}.

\begin{theorem}
By restricting on the class of separable C*-algebras of real rank zero,
we can recover  the ordered scaled total $\mathrm{K}$-theory from $\underline{V}_u$.
\end{theorem}

Hence, via \cite[Theorem 9.1]{DG}, 
the latticed total K-theory would distinguish the algebras in \cite{DL0,DE},
which can not be distinguished by the invariant related to the six-term exact sequence (see \cite[Remark 3.13]{ERZ}).

We are able to classify a large class of algebras in terms of the latticed total K-theory, of which the algebra can be
purely infinite, stably finite or even infinite with non-zero finite projections.
This provides an effective approach to the open question raised in \cite[5.1]{ERZ}. The class of C*-algebras we consider partially integrates and generalizes the types presented in \cite{EGKRT}.

Denote by $\mathcal{L}$ the collection of all the countably direct sums of the following algebras: A$\mathcal{HD}$ algebras of real rank zero,  Kirchberg algebras satisfying UCT and separable nuclear $\mathcal{O}_2$-stable algebras of real rank zero.

\begin{theorem}
Let $A_1,A_2,B_1,B_2\in \mathcal{L}$, assume $B_1,B_2$ are stable and $A_1,A_2$ are unital simple.
Then for any two unital extensions with trivial boundary maps
$$
0\to B_i\to E_i\to A_i \to 0,\quad i=1,2,
$$
we have $E_1\cong E_2$ if and only if
$
\underline{V}_u(E_1)\cong \underline{V}_u(E_2).
$
\end{theorem}




This paper is organized as follows.
In Section 2, we list some preliminaries for  total $\mathrm{K}$-theory with Dadarlat-Gong order,
extension theory.
In Section 3, we build the invariant $\underline{V}_u$---the latticed total $\mathrm{K}$-theory and show that total $\mathrm{K}$-theory  can be recovered from $\underline{V}_u$.
In Section 4, we talk about the ideal structure  of the latticed total K-theory, build the relative category together with corresponding functors and show the relations between them.
In Section 5, we prove the main classification result.


\section{Preliminaries}

\subsection{Terminologies and notations}
\begin{notion}\rm
  Let $A$ be a unital $\mathrm{C}^*$-algebra. $A$ is said to have stable rank one, written $sr(A)=1$, if the set of invertible elements of $A$ is dense. $A$ is said to have real rank zero, written $rr(A)=0$, if the set of invertible self-adjoint elements is dense in the set $A_{sa}$ of self-adjoint elements of $A$. If $A$ is not unital, let us denote the minimal unitization of $A$ by $\widetilde{A}$. A non-unital $\mathrm{C}^*$-algebra is said to have stable rank one (or real rank zero) if its unitization has stable rank one (or real rank zero).
\end{notion}

\begin{notion}\rm
  Let $A$ be a unital $\mathrm{C}^*$-algebra. Denote by $\mathcal{K}$ the compact operators on a separable infinite-dimensional Hilbert space and by $\mathcal{P}(A)$ the collection of all projections in $A$. Let $p,q\in \mathcal{P}(A)$. One says that $p$ is {\it Murray-von Neumann equivalent} to $q$ in $A$ and writes $p\sim_{\rm Mv} q$ if there exists $x\in A$ such that $x^*x=p$ and $xx^*=q$. We will write $p\lesssim_{\rm Mv} q$ if $p$ is equivalent to some sub-projection of $q$.

Let $V(A)$ be the Murray-von Neumann semigroup of $A$, i.e., $$V(A)=\{[p]_{\rm Mv}|\, p\in \mathcal{P}(A\otimes \mathcal{K})\},$$ where $[p]_{\rm Mv}$ is the class of  $p$. We use $[p]_{\mathrm{K}_0}$ to denote 
the class of $p$ in $\mathrm{K}_0(A)$.
\end{notion}

\begin{notion}\rm
$A$ is said to have {\it cancellation of projections}, if for any projections $p,q,e,f\in A$ with $pe=0$, $qf=0$, $e\sim_{\rm Mv} f$, and $p+e\sim_{\rm Mv} q+f$, then $p\sim_{\rm Mv} q$. $A$ has cancellation of projections if and only if $p\sim_{\rm Mv} q$ implies that there exists a unitary $u\in \widetilde{A}$ such that $u^*pu=q$ (see \cite[3.1]{L}). Every unital $\mathrm{C}^*$-algebra of stable rank one has cancellation of projections.

If $p$ is a projection in $A\otimes \mathcal{K}$, we denote by $I_p$ the ideal of $A\otimes \mathcal{K}$ generated by $p$. Note that if $p\sim_{\rm Mv} q$, then $I_p=I_q$, we will also use $I_{[p]_{\rm Mv}}$ to denote $I_p$. Let $I,J$ be ideals of $A$ with $I\subset J$. Denote by $\iota_{IJ}:\,I\to J$ the natural embedding map, we use $V(\iota_{IJ})$ to denote the induced map from $V(I)$ to $V(J)$.
\end{notion}

\begin{definition}\rm
We say a C*-algebra $A$ is $\sigma_{\rm p}$-{\it unital}, if $A$ has a countable approximate unit consisting of projections.
Denote by $C^*_{\sigma_{\rm p}}$ the category whose objects are separable $\sigma_{\rm p}$-unital C*-algebras and morphisms are *-homomorphisms.
\end{definition}

\begin{notion}\rm \label{AHD}{\rm (}\cite[2.2-2.3]{GJL}, \cite{GJL2}{\rm )}
We shall say a ${\rm C}^*$-algebra is an A$\mathcal{HD}$ algebra,
if it is an inductive limit of finite direct sums of algebras $M_n(\widetilde{\mathbb{I}}_p)$ and $PM_n(C(X))P$, where $$
\mathbb{I}_p=\{f\in M_p(C_0(0,1]):\,f(1)=\lambda\cdot1_p,\,1_p {\rm \,is\, the\, identity\, of}\, M_p\}
$$ is the Elliott-Thomsen dimension drop interval algebra
and $X$ is one of the following finite connected CW complexes: $\{pt\},~\mathbb{T},~[0, 1],~W_k.$ Here, $P\in M_n(C(X))$ is a projection and for $k\geq 2$, $W_k$ denotes the Moore space obtained by attaching the unit disk to the circle by a degree $k$-map, such as $f_k: \mathbb{T}\rightarrow \mathbb{T}$, $e^{{\rm i}\theta}\mapsto e^{{\rm i}k\theta}$. (This class of ${\rm C}^*$-algebras also play an important role in the classification of ${\rm C}^*$-algebras of ideal property \cite{GJLP1,GJLP2}.)

\end{notion}


\subsection{The category $\Lambda$ and total K-theory}

\begin{definition}\rm
Denote $\Lambda$ to be  the category whose objects are $\mathbb{Z}_2 \times \mathbb{Z}^{+}$ graded groups $G_0\oplus G_1\oplus\bigoplus_{n=1}^\infty (G_{0,n}\oplus G_{1,n})$ together with 
maps $\rho, \kappa$:
$$
\rho_{n}^j: G_{j} \rightarrow G_{j,n},\,\,
\kappa_{m n, m}^{j}: G_{j,m} \rightarrow G_{j,mn},\,\,
\kappa_{n, m n}^{j}: G_{j,mn} \rightarrow G_{j,m},\,\,j=0,1,
$$
and maps 
$$
\beta_{n}^{0}: G_{0,n} \rightarrow G_{1},\,\,\beta_{n}^{1}: G_{1,n} \rightarrow G_{0},
$$
satisfying the following exact sequences:
$$
\xymatrixcolsep{2pc}
\xymatrix{
{\,\,G_{0}\,\,} \ar[r]^-{\rho_n^0}
& {\,\,G_{0,n}\,\,}\ar[r]^-{\beta_n^0}
& {\,\,G_{1}\,\,} \ar[d]^-{\times n}
 \\
{\,\,G_{0}\,\,} \ar[u]^-{\times n}
& {\,\,G_{1,n}\,\,}\ar[l]^-{\beta_n^1}
& {\,\,G_{1}\,\,}\ar[l]^-{\rho_n^1} }
$$
and
$$
\xymatrixcolsep{2pc}
\xymatrix{
{\,\,G_{0,mn}\,\,} \ar[r]^-{\kappa_{n,mn}^0}
& {\,\,G_{0,n}\,\,}\ar[r]^-{\beta_{m,n}^0}
& {\,\,G_{1,m}\,\,} \ar[d]^-{\kappa_{mn,m}^1}
 \\
{\,\,G_{0,m}\,\,} \ar[u]^-{\kappa_{mn,m}^0}
& {\,\,G_{1,n}\,\,}\ar[l]^-{\beta_{m,n}^1}
& {\,\,G_{1,mn}\,\,}\ar[l]^-{\kappa_{n,mn}^1}, }
$$
where $\beta_{m,n}^0=\rho_n^{1}\circ\beta_n^0$ and $\beta_{m,n}^1=\rho_n^{0}\circ\beta_n^1.$ We will call such objects $\Lambda$-modules.

The morphisms in $\Lambda$ are graded group morphisms which are compatible with all the $\rho, \kappa,\beta$, and we will call such morphisms $\Lambda$-linear maps.
\end{definition}

\begin{notion}\label{def k-total}\rm
 (\cite[Section 4]{DG}) For $n\geq 2$, the mod-$n$ K-theory groups are defined by
$$\mathrm{K}_* (A;\mathbb{Z}_n)=\mathrm{K}_*(A\otimes C_0(W_n)),$$
where $\mathbb{Z}_n:=\mathbb{Z}/n\mathbb{Z}$, $W_n$ denotes the Moore space obtained by attaching the unit disk to the circle by a degree $n$-map.
The ${\rm C}^*$-algebra $C_0(W_n)$ of continuous functions vanishing at the base point is isomorphic to the mapping cone of the canonical map of degree $n$ from $C(\mathbb{T})$ to itself.
For $n=0$, we set ${\rm K}_{*}(A ; \mathbb{Z}_n)=$ ${\rm K}_{*}(A)$ and for $n=1, {\rm K}_{*}(A ; \mathbb{Z}_n)=0$.

For a $\mathrm{C}^{*}$-algebra $A$, one defines the total K-theory of $A$ by
$$
\underline{{\rm K}}(A)=\bigoplus_{n=0}^{\infty} {\rm K}_{*}(A ; \mathbb{Z}_n) .
$$
It is a $\mathbb{Z}_2 \times \mathbb{Z}^{+}$graded group. It was shown 
that the coefficient maps
$$
\begin{gathered}
\rho: \mathbb{Z} \rightarrow \mathbb{Z}_n, \quad \rho(1)=[1], \\
\kappa_{mn, m}: \mathbb{Z}_m \rightarrow \mathbb{Z}_{mn}, \quad \kappa_{m n, m}[1]=n[1], \\
\kappa_{n, m n}: \mathbb{Z}_{mn} \rightarrow \mathbb{Z}_n, \quad \kappa_{n, m n}[1]=[1],
\end{gathered}
$$
induce natural transformations
$$
\rho_{n}^{j}: {\rm K}_{j}(A) \rightarrow {\rm K}_{j}(A ; \mathbb{Z}_n),
$$
$$
\kappa_{m n, m}^{j}: {\rm K}_{j}(A ; \mathbb{Z}_m) \rightarrow {\rm K}_{j}(A ; \mathbb{Z}_{mn}),
$$
$$
\kappa_{n, m n}^{j}: {\rm K}_{j}(A ; \mathbb{Z}_{mn}) \rightarrow {\rm K}_{j}(A ; \mathbb{Z}_n) .
$$
The Bockstein operation
$$
\beta_{n}^{j}: {\rm K}_{j}(A ; \mathbb{Z}_n) \rightarrow {\rm K}_{j+1}(A)
$$
appears in the six-term exact sequence
$$
{\rm K}_{j}(A) \stackrel{\times n}{\longrightarrow} {\rm K}_{j}(A) \stackrel{\rho_{n}^{j}}{\longrightarrow} {\rm K}_{j}(A ; \mathbb{Z}_n) \stackrel{\beta_{n}^{j}}{\longrightarrow} {\rm K}_{j+1}(A) \stackrel{\times n}{\longrightarrow} {\rm K}_{j+1}(A)
$$
induced by the cofibre sequence
$$
A \otimes S C_{0}\left(\mathbb{T}\right) \longrightarrow A \otimes C_{0}\left(W_{n}\right) \stackrel{\beta}{\longrightarrow} A \otimes C_{0}\left(\mathbb{T}\right) \stackrel{n}{\longrightarrow} A \otimes C_{0}\left(\mathbb{T}\right),
$$
where $S C_{0}\left(\mathbb{T}\right)$ is the suspension algebra of $C_{0}\left(\mathbb{T}\right)$.

There is a second six-term exact sequence involving the Bockstein operations. This is induced by a cofibre sequence
$$
A \otimes S C_{0}\left(W_{n}\right) \longrightarrow A \otimes C_{0}\left(W_{m}\right) \longrightarrow A \otimes C_{0}\left(W_{m n}\right) \longrightarrow A \otimes C_{0}\left(W_{n}\right)
$$
and takes the form:
$$
{\rm K}_{j+1}(A ; \mathbb{Z}_n) \stackrel{\beta_{m, n}^{j+1}}{\longrightarrow} {\rm K}_{j}(A ; \mathbb{Z}_m) \stackrel{\kappa_{m n, m}^{j}}{\longrightarrow} {\rm K}_{j}(A ; \mathbb{Z}_{mn}) \stackrel{\kappa_{n, m n}^{j}}{\longrightarrow} {\rm K}_{j}(A ; \mathbb{Z}_n),
$$
where $\beta_{m, n}^{j}=\rho_{m}^{j+1} \circ \beta_{n}^{j}$.

Thus, $\underline{{\rm K}}(A)$ together with $\rho, \beta, \kappa$ becomes a $\Lambda$-module.

\end{notion}
\begin{notion}\rm
Let $A$ be a  C*-algebra with an approximate unit consisting of projections.
Recall that any ideal in $A\otimes \mathcal{K}$ is of the form $I\otimes\mathcal{K}$ for some ideal $I$ in $A$. If
$e$ is a projection in $A\otimes \mathcal{K}$, we denote by $I(e)$ the unique ideal of $A$ such that
$I(e)\otimes \mathcal{K}$ is the ideal of $A\otimes \mathcal{K}$ generated by $e$.

For $z\in \mathrm{K}_0^+(A)$,  we denote
by $\mathcal{I}(z)$ the set of all ideals $I(e)$ with the property that $z = [e] \in \mathrm{K}_0(A)$
for some projection $e$ in $A\otimes \mathcal{K}$. One may think of $\mathcal{I}(z)$ as being the ideal
support of the K-theory class $z$.
We say that a C*-algebra $ A$ has {\it  property} $(\mathcal{I})$ if $\mathcal{I}(z)$ is a singleton for
any $z \in \mathrm{K}_0^+(A)$. Equivalently, $A$ has property ($\mathcal{I}$) if and only if any two
projections in $A\otimes \mathcal{K}$ with the same K-theory class generate the same ideal.

Note that if $A$ has cancellation of projections or if $A$ is an A$\mathcal{HD}$-algebra,
then $A$ has property ($\mathcal{I}$).
If $I\otimes \mathcal{K}$ is an ideal of $A\otimes \mathcal{K}$, let us denote
$$\underline{\mathrm{K}}({I}\,|\,A)=:
\underline{\mathrm{K}}(\iota_{I})(\underline{\mathrm{K}}({I}))\subset \underline{\mathrm{K}}(A)$$
the image of $\underline{\mathrm{K}}({I})$
in $\underline{\mathrm{K}}({A})$. An element $x \in \underline{\mathrm{K}}({A})$ will be written in the form $x = (x_n^i)$ with
$x_n^i
\in \mathrm{K}_i(A;\mathbb{Z}_n)$, $i = 0, 1$. In particular, the component of $x$ in ${\rm K}_0(A)$ is
denoted by $x_0^0$.
\end{notion}

\begin{definition}\label{dg order}
\rm
Let $A$ be a separable C*-algebra.
We define $\underline{\mathrm{K}}(A)_+$ to be the subsemigroup of $\underline{\mathrm{K}}(A)$
generated by those elements
$x = (x_n^i) \in \underline{\mathrm{K}}(A)$ with the property that
$x_0^0\in \mathrm{K}_0^+(A)$ and
$$x \in \bigcup\, \{ \underline{\mathrm{K}}({I}\,|\,A)\mid \,I \in \mathcal{I}(x_0^0)\}.$$
Note that if $A$ is a simple C*-algebra, then
$$\underline{\mathrm{K}}(A)_+ = \{x\in \underline{\mathrm{K}}(A)\mid
 x_0^0\in \mathrm{K}_0^+(A)\}.$$
Suppose that $A$ and $B$ are C*-algebras and
$\phi :\, A\to B $ is a $*$-homomorphism. Then it is not hard to see that $\underline{\mathrm{K}}(\phi)(\underline{\mathrm{K}}(A)_+)
\subset \underline{\mathrm{K}}(B)_+$.

Note that our definition has a little difference from the original one in
\cite[Definition 4.6]{DG}. But if
the algebra $A$ has property $(\mathcal{I})$, they are the same. (Note that $A$ has cancellation
of projections or $A$ is simple implies $A$ has property ($\mathcal{I}$).)
We  still call this order structure {\bf  Dadarlat-Gong order}.

In particular, we may also denote
$$\mathrm{K}_*^+(A)=:\mathrm{K}_*(A)\cap \underline{\mathrm{K}}(A)_+,$$
where $\mathrm{K}_*(A)$ is identified with its natural image in $\underline{\mathrm{K}}(A)$ (This is also different from that given by Elliott  \cite{Ell} for the case of stable rank greater than one \cite{V}).
\end{definition}

Though we have a slight modification for the original definition, the following theorem stills holds.


\begin{theorem}{\rm (\cite[Porposition 4.8--4.9]{DG})}\label{ordertotal}
Suppose that $A$ has a countable approximate unit $(e_n)$ consisting
of projections. Then
\begin{enumerate}
  \item [(i)]$\underline{\mathrm{K}}(A)=\underline{\mathrm{K}}(A)_+-\underline{\mathrm{K}}(A)_+$;
  \item [(ii)]If $A$ is stably finite and has property $(\mathcal{I})$, $\underline{\mathrm{K}}(A)_+\cap\{-\underline{\mathrm{K}}(A)_+\} = \{0\}$, and hence,
$(\underline{\mathrm{K}}(A),\underline{\mathrm{K}}(A)_+)$ is an ordered group;
  \item [(iii)]For any $x\in \underline{\mathrm{K}}(A)$, there are positive integers $k$, $n$ such that $k[e_n]+x \in \underline{\mathrm{K}}(A)_+$.
\end{enumerate}
\end{theorem}

\subsection{Extension of C*-algebras}


\begin{proposition}{\rm (\cite[Proposition 4]{LR})}\label{lin inj}
Let $0\to B\to E\to A\to 0$
be an extension  of $C^*$-algebras. Let $\delta_{j}: \mathrm{K}_{j}(A) \rightarrow \mathrm{K}_{1-j}(B)$ for $j=0,1$, be the boundary  maps.
Assume that $A$ and $B$ have real rank zero. Then the following three conditions are equivalent :

{\rm (a)} $\delta_{0} \equiv 0$;

{\rm (b)} $rr(E)=0$;

{\rm (c)} all projections in $A$ are images of projections in $E$.
\end{proposition}

\begin{definition}\rm
We say an extension $e:0\to B\to E\to A\to 0$ has {\it  trivial boundary maps}, if
both boundary maps $\delta_0$ and $\delta_1$ are  zero.
\end{definition}

\begin{definition}\label{def ext}\rm
Let $A$, $B$ be ${\rm C}^*$-algebras and
$e\,:\, 0 \to B \to E \xrightarrow{\pi} A \to 0 $
be an extension of $A$ by $B$ with Busby invariant $\tau:\,A\to {\mathcal M}(B)/B$.

We say the extension $e$ is {\it essential}, if $\tau$ is injective.
We say $e$ is {\it full}, if for every
nonzero $a \in A$, $\tau (a)$ generates all of ${\mathcal M}(B)/B$ as a two-sided, closed ideal.
When $A$ is unital, we say  $e$ is {\it unital}, if $\tau$ is unital;
we say a unital extension $e$ is {\it unital trivial}, if there is
a unital $*$-homomorphism $\rho\, : A \to E$ such that $\pi\circ\rho = id_A$.

Note that if $A$ is unital simple and $B\neq 0$, any unital extension $e$ of $A$ by $B$ is essential and full.
\end{definition}
\begin{definition}\rm \label{strong tui cong}
Let
$e_i\,:\, 0 \to B \to E_i \to A \to 0 $
be two
extensions with Busby invariants $\tau_i$ for $i = 1, 2$. Then $e_1$ and $e_2$ are called {\it strongly unitarily equivalent}, denoted by $e_1
\sim_s e_2$,
if there exists a unitary $u \in {\mathcal M}(B)$ such that $\tau_2(a) = \pi(u)\tau_1(a)\pi(u)^*$ for all $a \in A$,
where $\pi:\,{\mathcal M}(B)\to {\mathcal M}(B)/B.$ It is well-known that $e_1\sim_s e_2$ implies $E_1\cong E_2$.

Assume that $B$ is stable. Fix an isomorphism of $\mathcal{K}$ with $M_2(\mathcal{K})$; this isomorphism induces an isomorphism $B \cong M_2(B)$ and hence isomorphisms ${\mathcal M}(B) \cong M_2({\mathcal M}(B))$ and  ${\mathcal M}(B)/B \cong M_2({\mathcal M}(B)/B)$. These isomorphisms are uniquely determined up to unitary equivalence.

Let $s_1,s_2$ be two isometries in ${\mathcal M}(B\otimes\mathcal{K})$ with $s_1s_1^*+s_2s_2^*=1$.
Define the addition $e_1 \oplus e_2$ to be the extension of $A$ by $B$ with Busby invariant
$$
\tau_1\oplus\tau_2:=\pi(s_1)\tau_1\pi(s_1)^*+\pi(s_2)\tau_2\pi(s_2)^*.
$$

We say that a unital extension $e$ with Busby invariant $\tau$ is {\it absorbing} if it is strongly unitarily equivalent to its sum with any  unital trivial extension, i.e.,  $e \sim_s e\oplus \sigma$ for any unital trivial extension $\sigma$.

\end{definition}



\begin{definition}\rm

Two unital extensions $e_1$ and $e_2$ are called {\it stably strongly unitarily equivalent}, denoted by $e_1
\sim_{ss} e_2$, if there exist unital trivial extensions $\sigma_1,\sigma_2$  such that $e_1\oplus\sigma_1 \sim_s e_2\oplus\sigma_2$.

If $A$ is unital,  denote  by $\mathrm{Ext}^u_{ss}(A, B)$ the set of stably strongly unitary equivalence classes of unital extensions of $A$ by $B$.
If $e_1,e_2$ are absorbing, then $e_1\sim_{ss}e_2$ coincides  with $e_1\sim_s e_2$.
\end{definition}

Denote by $\mathcal{N}$ the ``bootstrap'' category of \cite{RS}. We will  need the following ``UUCT" result, one can  see analogous versions under the circumstance in \cite{Sk1,Sk2}; see also \cite[Theorem 4.14]{GR} for the naturality.
\begin{theorem}{\rm (}\cite[Theorem 4.9]{W}{\rm )}\label{strong wei}
Suppose that $A$ is a unital separable nuclear $C^*$-algebra with $A\in\mathcal{N}$ and $B$ is separable stable $C^*$-algebra and has a countable approximate unit consisting of projections. Then
there is a short exact sequence of groups
$$
0 \to \mathrm{Ext}_{[1]}(\mathrm{K}_*(A), \mathrm{K}_*(B)) \to \mathrm{Ext}_{ss}^u(A, B) \to {\rm Hom}_{[1]}(\mathrm{K}_*(A), \mathrm{K}_*(B)) \to 0.
$$
\end{theorem}

\section{Latticed total K-theory}

\subsection{Properties of $V(A)$}
\begin{definition}\rm
 Let $A$ be a C*-algebra. We say a projection $p$ in $A$ is {\it infinite} if it is equivalent to a proper sub-projection of itself, and $p$ is {\it finite} otherwise. If $p$ has mutually orthogonal sub-projections $q_1$ and $q_2$ such that $p \sim_{\rm Mv} q_1 \sim_{\rm Mv} q_2$, then $p$ is said to be {\it properly infinite}. Equivalently, $p$ is properly infinite if $p$ is nonzero and $p \oplus p \lesssim_{\rm Mv} p$. 
A C*-algebra is infinite if it contains a  infinite projection.
\end{definition}

\begin{definition}\rm
We say $S$ is a {\it pre-ordered monoid} if $S$ is a monoid endowed with a pre-order relation $\leq$ which is compatible with the addition on $S$,
i.e., that $s_i\leq t_i$, $i=1,2,$ implies $s_1+s_2\leq t_1+t_2$.
Note that $s\leq t$ and $t\leq s$ do not necessarily imply $s=t$ in $S$.
Let $S,T$ be two pre-ordered monoids, we say $\alpha:S\to T$ is a {\it pre-order preserving monoid morphism} if $\alpha$ is additive and preserves the pre-order. 
\end{definition}

\begin{remark}\rm
For any ${\rm C}^{*}$-algebra $A$,
$V(A)$ is a pre-ordered monoid. 
If $A$ has an approximate unit consisting of projections (see \cite{Rbook}), we have
$$
{\rm Gr}(V(A))=\mathrm{K}_0(A),\quad \rho(V(A))=\mathrm{K}_0^+(A),
$$
where ${\rm Gr}(\cdot)$ is the Grothendieck construction and $\rho:\, V(A)\to \mathrm{K}_0(A)$ with $\rho([p]_{\rm Mv})=[p]_{{\rm K}_0}$.

As an example, we consider the Cuntz algebra $\mathcal{O}_4$ (see \cite{Cu alg}),
$$V(\mathcal{O}_4)=\{0\}\sqcup \mathbb{Z}_3,$$ and for $\bar{1},\bar{2}\in V(\mathcal{O}_4)$,
one may find that $\bar{1}\leq \bar{2}$ and $\bar{2}\leq \bar{1}$, but $\bar{1}\neq \bar{2}$. That is, $V(\mathcal{O}_4)$
is just a pre-ordered monoid, not an ordered monoid. (If $A$ has cancellation of projections, $V(A)$
is an ordered monoid.)
\end{remark}
In this paper, we will mainly consider $V(A)$ instead of $\mathrm{K}_0^+(A)$,
as $V(\cdot)$ does distinguish
if $A$ is infinite or has cancellation of projections.
\begin{proposition}\label{cancel or infinite}
Let  $A$ be a stable $C^*$-algebra with an approximate unit consisting of projections. Then

(i) If $A$ is infinite, then $V(A)$ doesn't have cancellation;

(ii) If  $A$ is of stable rank one, then $V(A)$ has cancellation.
\end{proposition}
\begin{proof}
We only show (i): Suppose  $V(A)$  has cancellation.
Since $A$ contains an infinite projection $p$, i.e., there exists a sub-projection
$p'$ of $p$ such that $p'\sim_{\rm Mv} p$, then we have $[p]_{\rm Mv}+[p-p']_{\rm Mv}=[p]_{\rm Mv}$ which implies $[p-p']_{\rm Mv}=0$ with $p-p'\neq 0$. This forms a contradiction.

\end{proof}

\begin{proposition}\label{V-injective}
Let $A$ be a $C^{*}$-algebra and suppose $I$ is an ideal of $A$. 
Then the induced map $V(\iota_{IA}):\,V(I)\to V(A)$ is an injective pre-order preserving monoid morphism.
\end{proposition}
\begin{proof}
It is obvious that $V(\iota_{IA})$ is a well-defined pre-order preserving monoid morphism.
We will show the injectivity.
For any projections $p,q$  in $I\otimes \mathcal{K}$, if $p\sim_{\rm Mv} q$  in $A\otimes \mathcal{K}$,
that is, there exists a $v\in A\otimes \mathcal{K}$ such that
$vv^*=p$ and $v^*v=q.$ Then we have $pv\in I\otimes \mathcal{K}$ and $(pv)(pv)^*=p$ and $(pv)^*(pv)=q$, hence, $p\sim_{\rm Mv} q$  in $I\otimes \mathcal{K}$.

\end{proof}
Let $I,J$ be ideals of $A$ with $I\subset J$. Thus, $V(\iota_{IJ}):\,V(I)\to V(J)$ is injective. 
We will always identify $V(I)$ as a subset of $V(J)$
through
$V(\iota_{IJ})$.

\subsection{Cuntz's KK picture}

\begin{notion}\label{cutnz new}\rm
(\cite{Cu1}) For each ${\rm C}^{*}$-algebra $A$, Cuntz defined the ${\rm C}^{*}$-algebra $Q A=A * A$ as the free product of $A$ by itself. $Q A$ is generated as a ${\rm C}^{*}$-algebra by $\{\iota(a), \bar{\iota}(a): a \in A\}$ with respect to the largest ${\rm C}^*$-norm, where $\iota$ and $\bar{\iota}$ denote the inclusions of two copies of $A$ into the free product. He also defined $q A$ as the ideal of $Q A$ generated by the differences $\{\iota(a)-\bar{\iota}(a)\mid a \in A\}$ and showed that for any  ${\rm C}^*$-algebras $A, B$,
$$
{\rm  KK}(A,B)=[qA,B\otimes \mathcal{K}],
$$
where $[qA,B\otimes \mathcal{K}]$ consists of all the homotopy classes of homomorphisms from $qA$ to $B\otimes \mathcal{K}$.

Let $\psi:\,A\to B$ be a homomorphism, denote by $SA$  the suspension algebra of $A$, i.e., $A\otimes C_0(0,1)$ and set $S\psi=\psi\otimes {\rm id}_{C_0(0,1)}.$ Then we have
 \begin{eqnarray*}
   {\rm K}_*(A;\mathbb{Z}_n) & = & {\rm K}_0(A;\mathbb{Z}_n)\oplus {\rm K}_1(A;\mathbb{Z}_n)  \\
     &= & {\rm  KK}(\mathbb{C},A\otimes  C_0(W_n))\oplus {\rm  KK}(\mathbb{C},SA\otimes  C_0(W_n))\\
     &= & [q\mathbb{C},A\otimes  C_0(W_n)\otimes\mathcal{K}]\oplus [q\mathbb{C},SA\otimes  C_0(W_n)\otimes\mathcal{K}].
  \end{eqnarray*}
\end{notion}
\begin{notion}\rm \label{theta and delta}
Let $A$ be a separable $\sigma_{\rm p}$-unital  ${\rm C}^*$-algebra and let $I$ and $J$ be ideals of $A\otimes \mathcal{K}$ such that $I\subset J$.
Denote by
$\delta_{IJ}$ the map ${\rm \underline{ K}}(\iota_{IJ})
: \underline{\mathrm{K}}(I)\to \underline{\mathrm{K}}(J)$, i.e.,
\begin{align*}
   & \delta_{IJ}([e],[{u}],\oplus_{n=1}^{\infty}([{s}_{n,0}],[{s}_{n,1}])) \\
   =& ({\rm K}_0(\iota_{IJ})([e]) ,{\rm K}_1(\iota_{IJ})([{u}]),\oplus_{n=1}^{\infty}
({\rm K}_0(\iota_{IJ};\mathbb{Z}_n)([{s}_{n,0}]),{\rm K}_1(\iota_{IJ};\mathbb{Z}_n)([{s}_{n,1}])))\\
=&([\iota_{IJ}(e)],[\iota^\sim_{IJ}(u)], \oplus_{n=1}^{\infty}
([(\iota_{IJ}\otimes {\rm id}_{C_0(W_n)})\circ s_{n,0}],[(S\iota_{IJ}\otimes {\rm id}_{C_0(W_n)})\circ s_{n,1}]))
\end{align*}
where $\iota^\sim_{IJ}$ is the unitized morphism of $\iota_{IJ}$, the maps
$${\rm K}_0(\iota_{IJ};\mathbb{Z}_n): {\rm K}_0(I;\mathbb{Z}_n)\rightarrow {\rm K}_0(J;\mathbb{Z}_n),\,\,{\rm K}_1(\iota_{IJ};\mathbb{Z}_n): {\rm K}_1(I;\mathbb{Z}_n)\rightarrow {\rm K}_1(J;\mathbb{Z}_n)$$
are induced by $\iota_{IJ}$, $[e]\in {\rm K}^+_0(I)$, $[{u}]\in {\rm K}_1(I)$, $[s_{n,0}]\in[q\mathbb{C},I\otimes  C_0(W_n)]$ and
$s_{n,1}\in [q\mathbb{C},SI\otimes  C_0(W_n)]$.

We will always identify
${\rm K}_1(I)\times\bigoplus_{n=1}^{\infty} {\rm K}_* (I; \mathbb{Z}_n)$ with its natural image in
${\rm \underline{K}}(I)=\bigoplus_{n=0}^{\infty} {\rm K}_{*}(I ; \mathbb{Z}_n)$, i.e.,
each $([{u}],\oplus_{n=1}^{\infty}([{s}_{n,0}],[{s}_{n,1}]))$ is identified with $(0,[{u}],\oplus_{n=1}^{\infty}([{s}_{n,0}],[{s}_{n,1}]))\in {\rm \underline{K}}(I)$.
Then the restriction of $\delta_{IJ}$ also induces a map
${\rm K}_1(I)\times\bigoplus_{n=1}^{\infty} {\rm K}_* (I; \mathbb{Z}_n)\to{\rm K}_1(J)\times\bigoplus_{n=1}^{\infty} {\rm K}_* (J; \mathbb{Z}_n)$, which is  still denoted by $\delta_{IJ}$.

\end{notion}

\subsection{Latticed total K-theory}

\begin{definition}\rm

Let $A$ be a separable $\sigma_{\rm p}$-unital C*-algebra and  $p,e\in \mathcal{P}(A\otimes\mathcal{K})$. 
Let $u,v$ be unitary elements of $I_p^\sim$ and $I_e^\sim$, respectively, and
let
$$s_{n,0}:\,q\mathbb{C}\to I_p\otimes C_0(W_n),\quad
t_{n,0}:\,q\mathbb{C}\to I_e\otimes C_0(W_n)$$
and
$$s_{n,1}:\,q\mathbb{C}\to SI_p\otimes C_0(W_n),\quad t_{n,1}:\,q\mathbb{C}\to SI_e\otimes C_0(W_n)$$ be homomorphisms for $n=1,2,\cdots$.

We say $(p,u,\oplus_{n=1}^{\infty}(s_{n,0},s_{n,1}))\lesssim_T(e,v,\oplus_{n=1}^{\infty}(t_{n,0},t_{n,1}))$,
if
$$
p\lesssim_{\rm Mv} e,\quad [\iota_{I_pI_e}^\sim(u)]=[v]\in {\rm K}_1(I_e),
$$
$$
[(\iota_{I_pI_e}\otimes {\rm id}_{C_0(W_n)})\circ s_{n,0})]=[t_{n,0}]\in {\rm K}_0(I_e;\mathbb{Z}_n)$$and
$$
[(S\iota_{I_pI_e}\otimes {\rm id}_{C_0(W_n)})\circ s_{n,1}]=[t_{n,1}]\in {\rm K}_1(I_e;\mathbb{Z}_n),\quad n=1,2,\cdots,
$$
which is equivalent to
$$
p\lesssim_{\rm Mv} e,\quad \delta_{I_pI_e}([u],\oplus_{n=1}^{\infty}([s_{n,0}],[s_{n,1}]))=([v],\oplus_{n=1}^{\infty}([t_{n,0}],[t_{n,1}])),
$$
where $\iota_{I_pI_e}$ (see \ref{theta and delta}) is the canonical inclusion from $I_p$ to $I_e$. ($\iota_{I_pI_e}$ induces the map $\delta_{I_pI_e}$, i.e.,
${\rm \underline{K}}(\iota_{I_pI_e})=\delta_{I_pI_e}$.)

\end{definition}
It is not hard to check that the relation $\lesssim_T$  is reflexive and transitive.

\begin{definition}\label{natural def ltk}
\rm
Let $A$ be a separable $\sigma_{\rm p}$-unital  ${\rm C}^*$-algebra.
Denote $\mathfrak{V}(A)$ the set of all the tuples $(p,u,\oplus_{n=1}^{\infty}(s_{n,0},s_{n,1}))$, where $p\in \mathcal{P}(A\otimes\mathcal{K})$, $u\in \mathcal{U}(I_p^\sim)$, $s_{n,0}\in {\rm Hom}(q\mathbb{C}, I_p\otimes C_0(W_n))$ and  $s_{n,1}\in {\rm Hom}(q\mathbb{C}, SI_p\otimes C_0(W_n))$.

Now we define the equivalent relation $\sim_T$ on $\mathfrak{V}(A)$ (called the total Murray--von Neumann equivalence) as follows.

Given $[(p,u,\oplus_{n=1}^{\infty}(s_{n,0},s_{n,1}))],[(e,v,\oplus_{n=1}^{\infty}(t_{n,0},t_{n,1}))]\in \underline{V}(A)$, we say
$$
(p,u,\oplus_{n=1}^{\infty}(s_{n,0},s_{n,1}))\sim_T(e,v,\oplus_{n=1}^{\infty}(t_{n,0},t_{n,1})),
$$
if
$$
p\sim_{\rm Mv} e\quad{ \rm and}\quad ([u],\oplus_{n=1}^{\infty}([s_{n,0}],[s_{n,1}]))=([v],\oplus_{n=1}^{\infty}([t_{n,0}],[t_{n,1}])).
$$
Denote by $[(p,u,\oplus_{n=1}^{\infty}(s_{n,0},s_{n,1}))]$ the equivalent class of $(p,u,\oplus_{n=1}^{\infty}(s_{n,0},s_{n,1}))$. We construct the latticed total K-theory of $A$ as follows:
$$
\underline{V}(A):=\mathfrak{V}(A)/\sim_T=\{[(p,u,\oplus_{n=1}^{\infty}(s_{n,0},s_{n,1}))]\mid (p,u,\oplus_{n=1}^{\infty}(s_{n,0},s_{n,1}))\in \mathfrak{V}(A)\}.
$$

For any two $[(p,u,\oplus_{n=1}^{\infty}(s_{n,0},s_{n,1}))],[(e,v,\oplus_{n=1}^{\infty}(t_{n,0},t_{n,1}))]\in \underline{V}(A)$, we say $$[(p,u,\oplus_{n=1}^{\infty}(s_{n,0},s_{n,1}))]\leq [(e,v,\oplus_{n=1}^{\infty}(t_{n,0},t_{n,1}))],$$ if
$$(p,u,\oplus_{n=1}^{\infty}(s_{n,0},s_{n,1}))\lesssim_T (e,v,\oplus_{n=1}^{\infty}(t_{n,0},t_{n,1})).$$

The following addition is well-defined on  $\underline{V}(A)$:
$$
[(p,u,\oplus_{n=1}^{\infty}(s_{n,0},s_{n,1}))]+[(e,v,\oplus_{n=1}^{\infty}(t_{n,0},t_{n,1}))]
$$$$
:=[(p\oplus e,u\oplus v,\oplus_{n=1}^{\infty}(\underline{s_{n,0}+t_{n,0}},\overline{t_{n,0}+t_{n,1}})],
$$
where $$\underline{s_{n,0}+t_{n,0}}=((\iota_{I_pI_{p\oplus e}}\otimes {\rm id}_{C_0(W_n)})\circ s_{n,0})\oplus ((\iota_{I_eI_{p\oplus e}}\otimes {\rm id}_{C_0(W_n)})\circ t_{n,0})$$ and
$$\overline{s_{n,1}+t_{n,1}}=((S\iota_{I_pI_{p\oplus e}}\otimes {\rm id}_{C_0(W_n)})\circ s_{n,1}) \oplus ((S\iota_{I_eI_{p\oplus e}}\otimes {\rm id}_{C_0(W_n)})\circ t_{n,1}).$$
For convenience, we point out that we have
$$
[(p,u,\oplus_{n=1}^{\infty}(s_{n,0},s_{n,1}))]+[(e,v,\oplus_{n=1}^{\infty}(t_{n,0},t_{n,1}))]
=[(c,w,\oplus_{n=1}^{\infty}(r_{n,0},r_{n,1}))]\in \underline{V}(A),$$ if and only if
$
p\oplus e\sim_{\rm Mv}c
$
and
$$
\delta_{I_p I_c}([u],\oplus_{n=1}^{\infty}([s_{n,0}],[s_{n,1}])) +\delta_{I_e I_c}([v],\oplus_{n=1}^{\infty}([t_{n,0}],[t_{n,1}]))
=([w],\oplus_{n=1}^{\infty}([r_{n,0}],[r_{n,1}])).
$$

Note that $[(0,1_{\mathbb{C}},\oplus_{n=1}^{\infty}(0,0))]$ is the neutral element and we obtain a pre-ordered monoid $\underline{V}(A)$.
Particularly, denoting
$$
\underline{V}(A)_+:=\{x\in \underline{V}(A)|\,[(0,1_{\mathbb{C}},\oplus_{n=1}^{\infty}(0,0))]\leq x \},
$$
we have $V(A)\cong \underline{V}(A)_+$ as pre-ordered monoid naturally and we will always identify them with each other.
\end{definition}

\begin{remark}\label{v total psi}
Let $A$ and $B$ be separable $\sigma_{\rm p}$-unital ${\rm C}^*$-algebras and let $\psi:\, A\to B$ be a $*$-homomorphism.
We still denote $\psi\otimes {\rm id}_\mathcal{K}$ by $\psi$. Let $p\in \mathcal{P}(A\otimes \mathcal{K})$ and $e\in \mathcal{P}(B\otimes \mathcal{K})$, suppose that $\psi(I_p)\subset I_e$, where $I_p$ and $I_e$ are ideals of $A\otimes \mathcal{K}$ and $B\otimes \mathcal{K}$ generated by $p$ and $e$, respectively.
Denote the restriction map $\psi|_{I_p\to I_e}:\,I_p\to I_e$ by $\psi_{I_p I_e}$ and  denote $\psi_{I_pI_{e}}^\sim$ the unitized morphism of $\psi_{I_p I_e}$.

For any $[(p,u,\oplus_{n=1}^{\infty}(s_{n,0},s_{n,1}))]\in \underline{V}(A)$, set $\underline{V}(\psi)$ for $\psi$ as follows,
$$
\underline{V}(\psi)([(p,u,\oplus_{n=1}^{\infty}(s_{n,0},s_{n,1}))])$$
$$=[(\psi(p),\psi_{I_pI_{\psi(p)}}^\sim(u),\oplus_{n=1}^{\infty}
( (\psi_{I_pI_{\psi(p)}}\otimes{\rm id}_{ C_0(W_n)})\circ s_{n,0},(S\psi_{I_pI_{\psi(p)}}\otimes{\rm id}_{ C_0(W_n)})\circ s_{n,1}))],
$$
where
$$\psi_{I_pI_{\psi(p)}}^\sim(u)\in I_{\psi(p)}^\sim,$$
$$(\psi_{I_pI_{\psi(p)}}\otimes{\rm id}_{ C_0(W_n)})\circ s_{n,0} \in {\rm Hom}( q\mathbb{C},I_{\psi(p)}\otimes C_0(W_n))$$
and
$$(S\psi_{I_pI_{\psi(p)}}\otimes{\rm id}_{ C_0(W_n)})\circ s_{n,1}\in {\rm Hom}( q\mathbb{C}, SI_{\psi(p)}\otimes C_0(W_n)).$$

\subsection{The functor $\underline{V}(\cdot)$}$\quad\\$
${\quad}$ It is easily seen that $\underline{V}(\psi)$ is a well-defined  monoid morphism. Here we only check $\underline{V}(\psi)$ is pre-order preserving:

For $[(p,{u},\oplus_{n=1}^{\infty}({s}_{n,0},{s}_{n,1}))]\leq [(e,{v},\oplus_{n=1}^{\infty}({t}_{n,0},{t}_{n,1}))]$,
we have
$$
p\lesssim_{\rm Mv} e,\quad [\iota_{I_pI_e}^\sim(u)]=[v]\in {\rm K}_1(I_e),
$$
$$
[(\iota_{I_pI_e}\otimes {\rm id}_{C_0(W_n)})\circ s_{n,0})]=[t_{n,0}]\in {\rm K}_0(I_e;\mathbb{Z}_n)$$and
$$
[(S\iota_{I_pI_e}\otimes {\rm id}_{C_0(W_n)})\circ s_{n,1}]=[t_{n,1}]\in {\rm K}_1(I_e;\mathbb{Z}_n),\quad n=1,2,\cdots,
$$
which is equivalent with
$$
p\lesssim_{\rm Mv} e,\quad \delta_{I_pI_e}([u],\oplus_{n=1}^{\infty}([s_{n,0}],[s_{n,1}]))=([v],\oplus_{n=1}^{\infty}([t_{n,0}],[t_{n,1}])).
$$
Then we have $\psi (p)\lesssim_{\rm Mv} \psi(e)$ and the following  diagram is commutative,
$$
\xymatrixcolsep{3pc}
\xymatrix{
{\,\,I_p\,\,} \ar[d]_-{\iota_{I_pI_e}} \ar[r]^-{\psi_{I_p I_{\psi(p)}}}
& {\,\,I_{\psi(p)}\,\,}  \ar[d]^-{\iota_{I_{\psi(p)}I_{\psi(e)}}}
\\
{\,\, I_e\,\,} \ar[r]^-{\psi_{I_e I_{\psi(e)}}}
& {\,\,I_{\psi(e)} \,\,}  }
$$
where $\psi_{I_p I_{\psi(p)}}$ is the restriction map of $\psi$ from $I_p$ to $I_{\psi(p)}$.
This implies that
$$
[\iota_{I_{\psi(p)}I_{\psi(e)}}^\sim \circ \psi_{I_p I_{\psi(p)}}^\sim(u)]
=[\psi_{I_e I_{\psi(e)}}^\sim\circ \iota_{I_{p}I_{e}}^\sim(u)]
=[\psi_{I_e I_{\psi(e)}}^\sim(v)]\in {\rm K}_1(I_{\psi(e)}),
$$
\begin{align*}
 &[((\iota_{I_{\psi(p)}I_{\psi(e)}}\circ\psi_{I_p I_{\psi(p)}})\otimes {\rm id}_{C_0(W_n)})\circ s_{n,0}] \\
 =& [((\psi_{I_e I_{\psi(e)}}\circ\iota_{I_{p}I_{e}})\otimes {\rm id}_{C_0(W_n)})\circ s_{n,0}] \\
   =& [(\psi_{I_e I_{\psi(e)}}\otimes {\rm id}_{C_0(W_n)})\circ t_{n,0}]\in {\rm K}_0(I_e;\mathbb{Z}_n)
\end{align*}
and
\begin{align*}
 & [((S\iota_{I_{\psi(p)}I_{\psi(e)}}\circ S\psi_{I_p I_{\psi(p)}})\otimes {\rm id}_{C_0(W_n)})\circ s_{n,1}]\\
 = & [((S\psi_{I_e I_{\psi(e)}}\circ S\iota_{I_{p}I_{e}})\otimes {\rm id}_{C_0(W_n)})\circ s_{n,1}] \\
   =& [(S\psi_{I_e I_{\psi(e)}}\otimes {\rm id}_{C_0(W_n)})\circ t_{n,1}]\in {\rm K}_1(I_e;\mathbb{Z}_n),
\end{align*}
where $n=1,2,\cdots$, which means that
$$
\underline{V}(\psi)([(p,u,\oplus_{n=1}^{\infty}(s_{n,0},s_{n,1}))])\leq \underline{V}(\psi)([(e,v,\oplus_{n=1}^{\infty}(t_{n,0},t_{n,1}))]).
$$
That is, $\underline{V}(\psi)$ is a pre-order preserving map.

Moreover, $\underline{V}(\cdot)$ is a functor from the category $C^*_{\sigma_{\rm p}}$ to the category  of pre-ordered monoids. 
\end{remark}

\subsection{Link with total K-theory}

\begin{theorem}\label{Gr V total}
If $A$ is a C*-algebra with  a countable approximate unit $(e_k)$ consisting of projections, we have a natural isomorphism
$$\alpha_A:\,(\mathrm{Gr}(\underline{V}(A)),\rho(\underline{V}(A)))\to (\underline{\mathrm{K}}(A),\underline{\mathrm{K}}(A)_+),$$
where $\rho:\,\underline{V}(A)\to \mathrm{Gr}(\underline{V}(A))$ with $\rho(x)=[(x,0)]_{\mathrm{Gr}}$.

Moreover, $\alpha_{(\cdot)}$ yields a natural equivalence between $\mathrm{Gr}(\underline{V}(\cdot))$ and $\underline{\mathrm{K}}(\cdot)$,
where $\mathrm{Gr}(\underline{V}(\cdot))$ and $\underline{\mathrm{K}}(\cdot)$ are regarded as functors from the category $C^*_{\sigma_{\rm p}}$
to the category of $\Lambda$-modules.
\end{theorem}
\begin{proof}
Assume $A$ is stable.
Define a map
$$
\alpha:\,\underline{V}(A)\to \underline{\mathrm{K}}(A)_+
$$
$$
\alpha([(p,u,\oplus_{n=1}^{\infty}(s_{n,0},s_{n,1}))])=([p]_{\mathrm{K}_0},[u],\oplus_{n=1}^{\infty}([s_{n,0}],[s_{n,1}])).
$$
From the definition of $\underline{\rm K}(A)_+$ listed in \ref{dg order}, $\alpha$ is a well-defined surjective monoid morphism.
($\alpha$ is just surjective, not necessarily injective; our modification fo Dadalat-Gong order guarantees $\alpha(\underline{V}(A))\subset \underline{\mathrm{K}}(A)_+$.)

  Note that ${\rm {K}}_0^+(A)$ is a sub-cone of $ {\rm \underline{K}}(A)_+$ (Theorem \ref{ordertotal} (i)), it induces a pre-order on ${\rm \underline{K}}(A)_+$, that is, for  $ (f,\overline{f}),(g,\overline{g})\in{\rm \underline{K}}(A)_+$ ($f,g\in {\rm K}_0^+(A)$ and $\overline{f},\overline{g}\in {\rm K}_1(A)\times\oplus_{n=1}^{\infty}{\rm K}_*(A;\mathbb{Z}_n)$), we say
$
(f,\overline{f})\leq(g,\overline{g}),
$
if
$
  g-f\in {\rm {K}}_0^+(A)\,\,{\rm and} \,\,\overline{f}=\overline{g}.
$


Then $\alpha$ is a per-order preserving monoid morphism.

  Suppose that
  $$
  \alpha[(e,{u},\oplus_{n=1}^{\infty}({s}_{n,0},{s}_{n,1}))]
  =\alpha[(q,{v},\oplus_{n=1}^{\infty}({t}_{n,0},{t}_{n,1}))],
  $$
  which is
  $$
  ([e]_{\mathrm{K}_0},\delta_{I_eA}([{u}],\oplus_{n=1}^{\infty}( [{s}_{n,0}],[{s}_{n,1}])))=
  ([q]_{\mathrm{K}_0},\delta_{I_qA}([{v}],\oplus_{n=1}^{\infty}( [{t}_{n,0}],[{t}_{n,1}]))).
  $$
$[e]_{\mathrm{K}_0}=[q]_{\mathrm{K}_0}$ implies that there exists a projection $p\in A$ such that
$$e\oplus p\sim_{\rm Mv} q\oplus p.$$

  As $(e_k)$ is a countable approximate unit  consisting of projections,
  denoting $e_k'=e\oplus q\oplus p\oplus \bigoplus_{i=1}^k e_i\in M_{k+3}(A)\cong A$ for $k\in \mathbb{N}$, we have
$$
I_{e_1'}\to I_{e_2'}  \to I_{e_3'}  \to\cdots
$$
is an inductive system with $\lim\limits_{k\to +\infty}   I_{e_k'}=A$.

Note that only finite many entries of   $([e],\delta_{I_eA}([{u}],\oplus_{n=1}^{\infty}( [{s}_{n,0}],[{s}_{n,1}])))$ and
 $ ([q],\delta_{I_qA}([{v}],\oplus_{n=1}^{\infty}( [{t}_{n,0}],[{t}_{n,1}])))$ are non-zero.
There exists a large enough $k_0$ such that
$$
\delta_{I_eI_{e_{k_0}'}}([e]_{{\rm K}_0},[{u}],\oplus_{n=1}^{\infty}( [{s}_{n,0}],[{s}_{n,1}])))=
  \delta_{I_qI_{e_{k_0}'}}([q]_{{\rm K}_0},[{v}],\oplus_{n=1}^{\infty}( [{t}_{n,0}],[{t}_{n,1}])).
$$
Then
  $$
  [({e_{k_0}'}\oplus e, \iota^\sim_{I_eI_{e_{k_0}'}}({u}), \oplus_{n=1}^{\infty}((\iota_{I_eI_{e_{k_0}'}}\otimes {\rm id}_{C_0(W_n)})\circ{s}_{n,0},
  (S\iota_{I_eI_{e_{k_0}'}}\otimes {\rm id}_{C_0(W_n)})\circ{s}_{n,1}))]
  $$$$=
  [({e_{k_0}'}\oplus q, \iota^\sim_{I_eI_{e_{k_0}'}}({v}), \oplus_{n=1}^{\infty}((\iota_{I_qI_{e_{k_0}'}}\otimes {\rm id}_{C_0(W_n)})\circ{t}_{n,0},
  (S\iota_{I_qI_{e_{k_0}'}}\otimes {\rm id}_{C_0(W_n)})\circ{t}_{n,1}))],
  $$
  which is
  $$
  [({e_{k_0}'},1_\mathbb{C},\oplus_{n=1}^{\infty}(0,0))]+[(e,{u},
  \oplus_{n=1}^{\infty}({s}_{n,0},{s}_{n,1}))]
  $$$$=
  [({e_{k_0}'},1_\mathbb{C},\oplus_{n=1}^{\infty}(0,0))]+[(q,{v},
  \oplus_{n=1}^{\infty}({t}_{n,0},{t}_{n,1}))].
  $$

  This means that for any $x,y\in \underline{V}(A)$, if $\alpha(x)=\alpha(y)$, under the Grothendieck construction, the difference between $x$ and $y$ vanishes.

  Consider the natural map $\rho:\, \underline{V}(A)\to {\rm Gr}(\underline{V}(A))$. If $\rho(x)=\rho(y)$, there exists $z\in \underline{V}(A)$ such that
$x+z=y+z$, hence, $$\alpha(x)+\alpha(z)=\alpha(y)+\alpha(z).$$ By Theorem \ref{ordertotal} (i),
$\alpha(x),\alpha(y),\alpha(z)$ are elements in the group $\underline{\mathrm{K}}(A)$, we get $\alpha(x)=\alpha(y)$.

  Then we have the canonical isomorphisms
  $$
  {\rm Gr}(\underline{V}(A) )\cong{\rm Gr}(\alpha(\underline{V}(A))= {\rm Gr}({\rm \underline{K}}(A)_+)={\rm \underline{K}}(A)
  $$
 and
 $$
 ({\rm Gr}(\underline{V}(A)),\rho(\underline{V}(A))\cong({\rm \underline{K}}(A),{\rm \underline{K}}(A)_+).
 $$

  Since  ${\rm \underline{K}}(A)$ is a $\Lambda$-module, then ${\rm Gr}(\underline{V}(A))$ can also be regarded as $\Lambda$-module
  through the canonical isomorphism.
   (In particular, from the identification of $V(A)$ and $\underline{V}(A)_+$ (at the end of \ref{natural def ltk}), we have $({\rm Gr}({V}(A)),\rho({V}(A))\cong ({\rm {K}_0}(A),{\rm {K}_0^+}(A))$).

For each $A\in C^*_{\sigma_{\rm p}}$,
denote the above canonical isomorphism by
$$\alpha_A:\,\mathrm{Gr}(\underline{V}(A))\to \underline{\mathrm{K}}(A).$$
Let $B\in C^*_{\sigma_{\rm p}}$ and $\psi:\,A\rightarrow B$ be a $*$-homomorphism, by Remark \ref{v total psi}, we have the induced map
$\underline{V}(\psi): \underline{V}(A)\rightarrow  \underline{V}(B)$.

Together with the following natural commutative diagram,
 $$
\xymatrixcolsep{3pc}
\xymatrix{
{\,\,\mathrm{Gr}(\underline{V}(A))\,\,} \ar[d]_-{\mathrm{Gr}(\underline{V}(\psi))} \ar[r]^-{\alpha_A}
& {\,\,\underline{\mathrm{K}}(A)\,\,}  \ar[d]^-{\underline{\mathrm{K}}(\psi)}
\\
{\,\, \mathrm{Gr}(\underline{V}(B))\,\,} \ar[r]^-{\alpha_B}
& {\,\,\underline{\mathrm{K}}(B)\,\,}  }
$$
we have the family $(\alpha_A)_{A\in C^*_{\sigma_{\rm p}}}$ is the desired natural equivalence of functors from $\mathrm{Gr}(\underline{V}(\cdot))$ to $\underline{\mathrm{K}}(\cdot)$.

\end{proof}


\subsection{An alternative observation}

\begin{notion}\rm


    Denote by ${\rm Lat}_{\rm p}(A)$ the subset of ${\rm Lat}(A)$ consisting of all the ideals of $A$ that contains a full projection. Hence, any ideal in ${\rm Lat}_{\rm p}(A)$ is singly-generated by a projection.
Define $V_{\rm p}(I):=\{[e]_{\rm Mv}\in V(A)\,|\, I_e=I\}$. In other words, $V_{\rm p}(I)$ consists of the elements of $V(A)$ that are full in $V(I)$. Thus,
if an ideal $I$ doesn't contain any full projections, we have
$V_{\rm p}(I)=\varnothing$.
\end{notion}

With a similar argument in \cite[Definition 3.18, Theorem 3.19]{AL}, we have the following abstract description of latticed total K-theory.

\begin{theorem}\label{cu total def}
  Let $A$ be a  separable $\sigma_{\rm p}$-unital ${C}^*$-algebra. 
  Then we have a natural isomorphism of pre-ordered monoid as follows:
  $$
  \underline{V}(A)
  \cong\coprod_{I\in {\rm Lat}_{\rm p}(A\otimes \mathcal{K})} V_{\rm p}(I)\times {\rm K}_1(I)\times\bigoplus_{n=1}^{\infty} {\rm K}_* (I; \mathbb{Z}_n),
  $$
  where we equip the right side with addition and pre-order as follows:
  For any
  $$
  (x,\mathfrak{u},\oplus_{n=1}^{\infty}(\mathfrak{s}_{n,0},\mathfrak{s}_{n,1}))\in V_{\rm p}(I_x)\times {\rm K}_1(I_x)\times\bigoplus_{n=1}^{\infty} {\rm K}_* (I_x; \mathbb{Z}_n)
  $$
  and
$$
  (y,\mathfrak{v},\oplus_{n=1}^{\infty}(\mathfrak{t}_{n,0},\mathfrak{t}_{n,1}))\in V_{\rm p}(I_y)\times {\rm K}_1( I_y)\times\bigoplus_{n=1}^{\infty} {\rm K}_* (I_y; \mathbb{Z}_n),
  $$
then
$$
(x,\mathfrak{u},\oplus_{n=1}^{\infty}(\mathfrak{s}_{n,0},\mathfrak{s}_{n,1}))
+(y,\mathfrak{v},\oplus_{n=1}^{\infty}(\mathfrak{t}_{n,0},\mathfrak{t}_{n,1}))
$$
$$
=(x+y,\delta_{I_xI_{x+y}}(\mathfrak{u},\oplus_{n=1}^{\infty}(\mathfrak{s}_{n,0},\mathfrak{s}_{n,1}))
+\delta_{I_yI_{x+y}}(\mathfrak{v},\oplus_{n=1}^{\infty}(\mathfrak{t}_{n,0},\mathfrak{t}_{n,1})))
$$
and
$$
(x,\mathfrak{u},\oplus_{n=1}^{\infty}(\mathfrak{s}_{n,0},\mathfrak{s}_{n,1}))
\leq(y,\mathfrak{v},\oplus_{n=1}^{\infty}(\mathfrak{t}_{n,0},\mathfrak{t}_{n,1})),
$$
if
$$\,\,x\leq y~~{\rm in}~~ V(A)\,\,{\rm and}\,\,\delta_{I_xI_y}
(\mathfrak{u},\oplus_{n=1}^{\infty}(\mathfrak{s}_{n,0},\mathfrak{s}_{n,1}))=(\mathfrak{v},
\oplus_{n=1}^{\infty}(\mathfrak{t}_{n,0},\mathfrak{t}_{n,1})),
$$
where $\delta_{I_xI_y}$ is the natural map $\underline{\rm K}( I_x)\rightarrow \underline{\rm K}( I_y)$. 
Note that
 $(\mathfrak{u},\oplus_{n=1}^{\infty}(\mathfrak{s}_{n,0}
,\mathfrak{s}_{n,1}))$ and $(\mathfrak{v}, \oplus_{n=1}^{\infty}(\mathfrak{t}_{n,0},\mathfrak{t}_{n,1}))$ are identified as $(0,\mathfrak{u},\oplus_{n=1}^{\infty}(\mathfrak{s}_{n,0},\mathfrak{s}_{n,1}))\in {\rm \underline{K}}(I_x)$ and $(0,\mathfrak{v},\oplus_{n=1}^{\infty}(\mathfrak{t}_{n,0},\mathfrak{t}_{n,1}))\in {\rm \underline{K}}(I_y)$, respectively.
\end{theorem}

\section{Ideal and Latticed morphism}

\subsection{Ideals of  the invariant}

\begin{definition}\label{PD} {\rm (}\cite[Definition 3.3]{L2}{\rm )}~\rm
Let $S$ be a pre-ordered monoid. We say that $S$ is {\it positively directed}, if for any $x \in S$, there exists $p_{x} \in S$ such that $x+p_{x} \geq 0$. 

\end{definition}
\begin{definition}\label{PS}{\rm (}\cite[Definition 3.9]{L2}{\rm )}~\rm
Let $S$ be a positively directed pre-ordered monoid. Let $M$ be
a submonoid of $S$. We say $M$ is an {\it ideal} of $S$ if it satisfies the following.
\begin{enumerate}
\item[(i)] $M$ is a positively directed pre-ordered monoid.
\item[(ii)]  For any $x \in S$, if $(x + P_x) \cap M \neq \varnothing$, then $x \in M$, where $P_x := \{y \in S\mid
x + y\geq 0\}.$
\end{enumerate}

In particular, suppose $S$ is a pre-ordered monoid satisfying that for any $s\in S$, we have $s\geq 0$. Then the
following are equivalent:
\begin{enumerate}
\item[(1)] A submonoid $M$ is an ideal of $S$;
\item[(2)] For any $x\in S$, if $x+y\in M$ for some $y\in S$, we have $x\in M$,
(i.e.,
$M$ is (algebraic) hereditary in $S$).
\end{enumerate}
\end{definition}
The following proposition is a direct corollary of \cite[Theorem 2.3]{Z}
\begin{proposition}\label{Zhang lemma}
If $A$ is a separable $C^*$-algebra of real rank zero,
there is a natural isomorphism:
${\rm Lat}(A)\cong {\rm Lat}(V(A)).$
\end{proposition}
We also raise the following version:
\begin{theorem}\label{total lat iso}
Let $A$ be a separable ${C}^*$-algebra of real rank zero. Denote by
${\rm Lat}(\underline{V}(A))$ the collection
of all the ideals of  $\underline{V}(A)$.
Then the map
\begin{eqnarray*}
  \Phi: {\rm Lat} (A) &\to & {\rm Lat}(\underline{V}(A)), \\
  I &\mapsto& \underline{V}(I)
\end{eqnarray*}
is an isomorphism of complete lattices. 
In particular, $A$ is simple if
and only if $\underline{V}(A)$ is simple.
\end{theorem}
\begin{proof}
Assume $A$ is stable and we identify $V(A)$ with $\underline{V}(A)_+$. By Theorem \ref{cu total def}, $(x,\mathfrak{u},\bigoplus_{n=1}^{\infty}(\mathfrak{s}_{n,0},\mathfrak{s}_{n,1}))
\in\underline{V}(A)$ belongs to $\underline{V}(I)$ if and only if $x \in V(I)$.

Step 1:  we show $\Phi$ is well defined,
that is, for any ideal $I$ of $A$, $\underline{V}(I)$ is an ideal of $\underline{V}(A)$.

By Proposition \ref{Zhang lemma}, we have $V(I)$ is an ideal of $V(A)$, i.e., $V(I)$ is hereditary in $V(A)$.
Given any   $(x,\mathfrak{u},\bigoplus_{n=1}^{\infty}(\mathfrak{s}_{n,0},\mathfrak{s}_{n,1}))\in\underline{V}(I)$, by Theorem  \ref{cu total def}, we have $x \in V(I)$.
Then $(x,-\mathfrak{u},\bigoplus_{n=1}^{\infty}(-\mathfrak{s}_{n,0},-\mathfrak{s}_{n,1}))\in\underline{V}(I)$
and
$$\textstyle
(x,\mathfrak{u},\bigoplus\limits_{n=1}^{\infty}(\mathfrak{s}_{n,0},\mathfrak{s}_{n,1}))
+(x,-\mathfrak{u},\bigoplus\limits_{n=1}^{\infty}(-\mathfrak{s}_{n,0},-\mathfrak{s}_{n,1}))=(2x,0,\bigoplus\limits_{n=1}^{\infty}(0,0))\in\underline{V}(I)
$$
which is positive. By Definition \ref{PD}, $\underline{V}(I)$ is positively directed.

Given any   $(x,\mathfrak{u},\bigoplus_{n=1}^{\infty}(\mathfrak{s}_{n,0},\mathfrak{s}_{n,1}))\in\underline{V}(A)$.
If there exists an element $(y,\mathfrak{v},\bigoplus_{n=1}^{\infty}(\mathfrak{t}_{n,0},\mathfrak{t}_{n,1}))\in\underline{V}(A)$
such that
$$\textstyle
(x,\mathfrak{u},\bigoplus\limits_{n=1}^{\infty}(\mathfrak{s}_{n,0},\mathfrak{s}_{n,1}))
+(y,\mathfrak{v},\bigoplus\limits_{n=1}^{\infty}(\mathfrak{t}_{n,0},\mathfrak{t}_{n,1}))
=(x+y,0,\bigoplus\limits_{n=1}^{\infty}(0,0))\in \underline{V}(I)
$$
which is positive.  Then we have $x+y\in V(I)$, by Definition \ref{PS} (2), $x\in V(I)$.
That is,  $(x,\mathfrak{u},\bigoplus_{n=1}^{\infty}(\mathfrak{s}_{n,0},\mathfrak{s}_{n,1}))\in \underline{V}(I),$ by Definition \ref{PS} (ii), we have $\underline{V}(I)$   is an ideal of $\underline{V}(A)$.

Step 2: Given any ideal $M$ of $\underline{V}(A)$, we will show that there exists an ideal $I$ of $A$ such that $M=\underline{V}(I)$.

Denote $M_+:=\{s\in M|\, s\geq 0 \}$.
By  Theorem  \ref{cu total def}, any $s\in M_+$
has the form $(x,0,\bigoplus_{n=1}^{\infty}(0,0))$, thus,
$
M_+\subset {V}(A).
$
By Definition \ref{PD}, it is easy to see that $M_+$ is an ideal of ${V}(A)$, and by Proposition \ref{Zhang lemma},
there is a unique ideal $I$ of $A$ with $V(I)=M_+$.
Now we need to check $\underline{V}(I)=M$.

For any $(x,\mathfrak{u},\bigoplus_{n=1}^{\infty}(\mathfrak{s}_{n,0},\mathfrak{s}_{n,1}))\in \underline{V}(I),$
then $(x,-\mathfrak{u},\bigoplus_{n=1}^{\infty}(-\mathfrak{s}_{n,0},-\mathfrak{s}_{n,1}))\in\underline{V}(I)$
and
$$\textstyle
(x,\mathfrak{u},\bigoplus\limits_{n=1}^{\infty}(\mathfrak{s}_{n,0},\mathfrak{s}_{n,1}))
+(x,-\mathfrak{u},\bigoplus\limits_{n=1}^{\infty}(-\mathfrak{s}_{n,0},-\mathfrak{s}_{n,1}))
$$
$$\textstyle
=(2x,0,\bigoplus\limits_{n=1}^{\infty}(0,0))\in{V}(I)=M_+\subset M.
$$
That $M$ is an ideal of $\underline{V}(A)$ implies $(x,\mathfrak{u},\bigoplus_{n=1}^{\infty}(\mathfrak{s}_{n,0},\mathfrak{s}_{n,1}))\in M.$
That is, $\underline{V}(I)\subset M$.

Similarly, we can obtain $ M \subset\underline{V}(I)$, then $\underline{V}(I)=M$.

Indeed, Step 2 also implies that $\underline{V}(I)$ is one-to-one corresponding to $V(I)$.
Thus combining both the two steps with Proposition \ref{Zhang lemma}, we have
 the map
\begin{eqnarray*}
  \Phi: {\rm Lat} (A) &\to & {\rm Lat}(\underline{V}(A)), \\
  I &\mapsto& \underline{V}(I)
\end{eqnarray*}
is an isomorphism of complete lattices.

\end{proof}

For convenience, we write the inverse map as follows:
$$
\Psi : {\rm Lat}(\underline{V}(A)) \to {\rm Lat}(A),$$
$$\textstyle
J\mapsto \{y \in A \mid\, yy^* \leq p ~~{\rm for ~~some ~~projection}~~ p ~~{\rm in}~~ A ~~{\rm with}~~ [p]_{\rm Mv}\in J_+\}$$
where $J_+=\{s\in J \mid 0\leq s\}$ is the positive cone of $J$.

\subsection{The category $\underline{\mathcal{V}}$}

\begin{notion}\rm
Let $S$ be a positively directed pre-ordered monoid.
We say $S$ is a $\underline{\mathcal{V}}$-cone,
if

(i) the collection of all the  ideals of $S$ forms a lattice;

(ii) for any ideals $I_1,I_2$ of $S$, we have following natural commutative diagram in $\Lambda$-category
$$
\xymatrixcolsep{2pc}
\xymatrix{
{\,\,\mathrm{Gr}(I_1)\,\,} \ar[r]^-{}
& {\,\,\mathrm{Gr}(I_1+I_2)\,\,}
 \\
{\,\,\mathrm{Gr}(I_1\cap I_2)\,\,} \ar[r]_-{}\ar[u]^-{}
& {\,\,\mathrm{Gr}(I_2)\,\,.}\ar[u]^-{}}
$$

\end{notion}
\begin{proposition}\label{v(a) Lambda cong}
Let $A$ be a separable ${C}^*$-algebra of real rank zero, $\underline{V}(A)$ is a $\underline{\mathcal{V}}$-cone.
\end{proposition}
\begin{proof}
Assume $A$ is stable. By Proposition \ref{total lat iso}, all the  ideals of $\underline{V}(A)$ forms a lattice and
for any ideals $J_1,J_2$ of $\underline{V}(A)$, there exists ideals $I_1,I_2$ of $A$ such that
$\underline{V}(I_i)=J_i,\quad i=1,2.$

Then we obtain a commutative diagram in the category  of separable ${\rm C}^*$-algebras of real rank zero
as follows
$$
\xymatrixcolsep{2pc}
\xymatrix{
{\,\,I_1\,\,} \ar[r]^-{}
& {\,\,I_1+ I_2\,\,}
 \\
{\,\,I_1\cap  I_2\,\,} \ar[r]_-{}\ar[u]^-{}
& {\,\,I_2\,\,}\ar[u]^-{},}
$$
where the corresponding maps are the natural inclusions.
By Theorem \ref{Gr V total}, we have the following commutative diagram
$$
\xymatrixcolsep{2pc}
\xymatrix{
{\,\,{\rm Gr}(\underline{V}(I_1))\cong \underline{\mathrm{K}}(I_1)\,\,} \ar[r]^-{}
& {\,\,{\rm Gr}(\underline{V}(I_1+ I_2))\cong \underline{\mathrm{K}}(I_1+I_2)\,\,}
 \\
{\,\,{\rm Gr}(\underline{V}(I_1\cap  I_2))\cong \underline{\mathrm{K}}(I_1\cap I_2)\,\,} \ar[r]_-{}\ar[u]^-{}
& {\,\,{\rm Gr}(\underline{V}(I_2))\cong \underline{\mathrm{K}}(I_2)\,\,}\ar[u]^-{}.}
$$

\end{proof}

\begin{definition}\label{def lambda cat}\rm
Define the $\underline{\mathcal{V}}$-category   as follows:
$$
{\rm Ob}(\underline{\mathcal{V}})=\{\,S\,|\,S\,\,{\rm is\,\, a\,\, \underline{\mathcal{V}}\text{-}cone}\,\};
$$
let $X,Y\in {\rm Ob}(\underline{\mathcal{V}})$, we say the map $\phi:\,X\to Y$ is a $\underline{\mathcal{V}}$-morphism if $\phi$ satisfies the following two conditions:

(1) $\phi$ is a  pre-order preserving monoid morphism;

(2) For any  ideal couples $(I_1,J_1)$, $(I_2,J_2)$ of $(X,Y)$ with $J_i$ as the ideal in $Y$ containing $\phi(I_i)$ for each $i=1,2,$ we have the following diagram commute naturally in $\Lambda$-category.
\begin{displaymath}
\xymatrixcolsep{0.8pc}
\xymatrix{
{\rm Gr}(I_1) \ar[rr]^-{}\ar[dr]^-{} && {\rm Gr}(I_1+ I_2) \ar[dr]^-{} \\
&{\rm Gr}(J_1) \ar[rr]^-{}&&
{\rm Gr}(J_1+ J_2) \\
&{\rm Gr}(J_1\cap J_2)  \ar[rr]^-{} \ar[u]^-{}&&
{\rm Gr}(J_2)  \ar[u]^-{}\\
{\rm Gr}(I_1\cap I_2) \ar[rr]^-{}\ar[ur]^-{} \ar[uuu]^-{} && {\rm Gr}(I_2) \ar[ur]^-{}\ar[uuu]^-{}
}
\end{displaymath}
\end{definition}






\begin{definition}\label{cone scale}\rm
Given a  $\underline{\mathcal{V}}$-cone $S$ and $\Delta\subset S_+=\{s\in S\mid\, 0\leq s\}$, we say
$\Delta$ is a {\it scale} of $S$ if $\Delta$ is upper directed, hereditary, full subset of $S_{+}$, i.e.,

(i) $\forall\, x_1, x_2 \in \Delta,\, \exists\, x \in \Delta: x_1 \leq x, x_2 \leq x$,

(ii) $\forall\, x \in S_{+},\, \forall\, y \in \Delta: x \leq y \Rightarrow x \in \Delta$,

(iii) $\forall \, x \in S_{+},\, \exists\, y \in \Delta,\, \exists\, k \in \mathbb{N}: x \leq k y$.

\end{definition}
\begin{proposition}
If $A$ is  a ${C}^*$-algebra 
and $A$ has a countable approximate unit $(e_k)$ consisting of projections,  then
$\Delta(A)=\{[(p,1_\mathbb{C},\oplus_{n=1}^{\infty}(0,0))]\,|\,p\in A\}$ is a scale of $\underline{V}(A)$. (Note that $\Delta(A)$ can be also identified as the
subset
$\{[p]_{\rm Mv}\,|\,p\in A\}$ of $V(A)$.)
\end{proposition}
\begin{proof} We check the conditions in  Definition \ref{cone scale}.

(i): for any $ x_1, x_2 \in \Delta(A)$,
we lift $x_1,x_2$ to projections $p_1,p_2$ in $A$, then there exists a large enough $k_0\geq 0$ such that
$$
\|p_i-e_{k_0}p_i\|< 1, \quad i=1,2.
$$
Thus, we have $p_i \lesssim_{\rm Mv}e_{k_0}$, $i=1,2$.

(ii):  trivial.

(iii): for any  $x \in \underline{V}(A)_+=V(A)$, there exists an integer $n\geq 0$ such that $x$ can be lift
to a projection $p$ in $M_n(A)$. Since $(e_k\otimes 1_n)_k$ forms an approximate unit of $M_n(A)$ and
there exists a large enough $k_1\geq 0$ such that
$$
\|p-(e_{k_1}\otimes 1_n)\cdot p\|< 1.
$$
Then we obtain $p\lesssim_{\rm Mv}   e_{k_1}\otimes 1_n$, that is, $[p]_{\rm Mv}\leq n\cdot[e_{k_1}]_{\rm Mv}$.

\end{proof}
\subsection{Recover the total K-theory}
\begin{definition}\rm
Define $\underline{\mathcal{V}}_u$
the category whose objects are  all tuples $(S,\Delta_S)$ with $S$ is a $\underline{\mathcal{V}}$-cone and $\Delta_S$ is a scale of $S$ and morphisms are  $\underline{\mathcal{V}}$-morphisms preserving the scale. We say $\phi: (S,\Delta_S)\to(T,\Delta_T)$ is a $\underline{\mathcal{V}}_u$-morphism if
 $\phi:S\to T$ is a $\underline{\mathcal{V}}$-morphism and $\phi(\Delta_S)\subset \Delta_T$.





\end{definition}

\begin{proposition}\label{scaled total v recover}
The assignment
  \begin{eqnarray*}
    \underline{V}_u:\,  C_{rr0}^* & \rightarrow & \underline{\mathcal{V}}_u  \\
    A &\mapsto & (\underline{V}(A),\Delta(A)) \\
    \psi &\mapsto & \underline{V}(\psi)
  \end{eqnarray*}
  from the category of  separable ${C}^*$-algebras of real rank zero to the category $\underline{\mathcal{V}}_u$ is a covariant functor.
\end{proposition}
\begin{proof}
We may assume that $A,B$ are separable ${\rm C}^*$-algebras of real rank zero. By Proposition \ref{v(a) Lambda cong} and Definition \ref{cone scale},
we have $(\underline{V}(A),\Delta(A))$ is an object in the category $\underline{\mathcal{V}}_u$.

For any homomorphism $\psi:\, A\to B$,
recall the construction in Remark \ref{v total psi}, it is trivial that
$\underline{V}(\psi)(\Delta(A))\subset \Delta(B)$.
Let us show that $\underline{V}(\psi)$ is a $\underline{\mathcal{V}}$-morphism.
(In Remark \ref{v total psi}, it has been shown that $\underline{V}(\psi)$ is pre-order preserving monoid morphism.)

For any  ideal couples $(I_1,J_1)$, $(I_2,J_2)$ of $(\underline{V}(A),\underline{V}(B))$ with $J_i$ as the ideal in $\underline{V}(B)$ containing $\underline{V}(\psi)(I_i)$ for each $i=1,2,$
by Theorem \ref{total lat iso},
there exist
ideal couples $(I_1',J_1')$, $(I_2',J_2')$ of $(A,B)$ with $J_i'$ as the ideal in $B$ containing $\psi(I_i')$
such that
$
\underline{V}(I_i')=I_i,\,\, \underline{V}(J_i')=J_i,
$ for each $i=1,2$.

This implies the following commutative diagram in the category $C_{rr0}^*$.
$$
\xymatrixcolsep{1.3pc}
\xymatrix{
I_1' \ar[rr]^-{}\ar[dr]^-{} && I_1'+ I_2' \ar[dr]^-{} \\
&J_1' \ar[rr]^-{}&&
J_1'+ J_2' \\
&J_1'\cap J_2' \ar[rr]^-{}\ar[u]^-{}&&
J_2'  \ar[u]^-{}\\
I_1'\cap I_2' \ar[rr]^-{}\ar[ur]^-{}\ar[uuu]^-{} && I_2' \ar[ur]^-{}\ar[uuu]^-{}
}
$$
Then the following induced diagram commutes naturally in $\Lambda$-category.
$$
\xymatrixcolsep{1.3pc}
\xymatrix{
\underline{\mathrm{K}}(I_1') \ar[rr]^-{}\ar[dr]^-{} && \underline{\mathrm{K}}(I_1'+ I_2') \ar[dr]^-{} \\
&\underline{\mathrm{K}}(J_1') \ar[rr]^-{}&&
\underline{\mathrm{K}}(J_1'+ J_2') \\
&\underline{\mathrm{K}}(J_1'\cap J_2') \ar[rr]^-{} \ar[u]^-{}&&
\underline{\mathrm{K}}(J_2')  \ar[u]^-{}\\
\underline{\mathrm{K}}(I_1'\cap I_2') \ar[rr]^-{}\ar[ur]^-{} \ar[uuu]^-{}&& \underline{\mathrm{K}}(I_2') \ar[ur]^-{}\ar[uuu]^-{}
}
$$
That is, by Theorem \ref{Gr V total}, the following diagram commutes naturally in $\Lambda$-category, which concludes the proof.
$$
\xymatrixcolsep{0.3pc}
\xymatrix{
{\rm Gr}(\underline{V}(I_1')) \ar[rr]^-{}\ar[dr]^-{} && {\rm Gr}(\underline{V}(I_1'+ I_2')) \ar[dr]^-{} \\
&{\rm Gr}(\underline{V}(J_1')) \ar[rr]^-{}&&
{\rm Gr}(\underline{V}(J_1'+ J_2')) \\
&{\rm Gr}(\underline{V}(J_1'\cap J_2')) \ar[rr]^-{}\ar[u]^-{}&&
{\rm Gr}(\underline{V}(J_2'))  \ar[u]^-{}\\
{\rm Gr}(\underline{V}(I_1'\cap I_2')) \ar[rr]^-{}\ar[ur]^-{}\ar[uuu]^-{} && {\rm Gr}(\underline{V}(I_2')) \ar[ur]^-{}\ar[uuu]^-{}
}
$$

\end{proof}

As expected, we can recover the total K-theory from $\underline{V}$,
which is a direct corollary of Theorem \ref{Gr V total}. 
\begin{proposition}\label{recover prop}
  The assignment
\begin{eqnarray*}
  H:\,  \underline{\mathcal{V}}_u  &\rightarrow& \Lambda_u\\
  (S,\Delta) &\mapsto & ({\rm Gr}(S),\rho(S),
  \rho(\Delta))\\
  \phi &\mapsto & {\rm Gr}(\phi)
\end{eqnarray*}
is a functor, where $\rho:\, S\to {\rm Gr}(S)$ is the natural map $(\rho(x)=[(x,0)]_{\rm Gr})$ and $\Lambda_u$ denotes the category whose objects are ordered scaled $\Lambda$-modules and morphisms are pre-order preserving scaled $\Lambda$-morphisms  in the standard sense. 

 The functor $H$ yields a natural equivalence $H\circ \underline{V}_u\cong(\underline{\mathrm{K}}(\cdot),\underline{\mathrm{K}}(\cdot)_+,\Sigma(\cdot))$, which means, for any  separable {C}*-algebras $A$ of real rank zero, 
we have a natural isomorphism:
$$H\circ \underline{V}_u(A)\cong(\underline{\mathrm{K}}(A),\underline{\mathrm{K}}(A)_+,\Sigma(A)),
$$
where $\Sigma(A)=:\{[p]_{\mathrm{K}_0}|\, p {\rm ~~is ~~a ~~projection~~ in} ~~ A\}.$

\end{proposition}
Note that from Proposition \ref{scaled total v recover}, the condition (2) in Definition \ref{def lambda cat} is a natural one.
Moreover, we shall still emphasize that
$$\underline{V}(A)\cong \underline{V}(B)$$
doesn't only mean there exists a
graded pre-order preserving monoid isomorphism from $\underline{V}(A)$ to $\underline{V}(B)$, the extra requirement (condition (2) in Definition \ref{def lambda cat}) on the lattice structure is necessary.

\begin{example}
Let $A=E\otimes\mathcal{K},B=E'\otimes\mathcal{K}$, where $E,E'$ are C*-algebras in \cite[Theorem 3.3]{DEb}. Thus $A,B$ are A$\mathcal{HD}$ algebras of real rank zero and
there exists an ordered $\underline{\mathrm{K}}_{\langle\beta\rangle}$-isomorphism $\gamma$ (i.e.,
$\gamma$ is graded ordered group isomorphism commutating with all $\rho,\kappa$ maps, but not commutating with some $\beta$ maps) such that  
$
\underline{\mathrm{K}}_{\langle\beta\rangle}(A)\cong \underline{\mathrm{K}}_{\langle\beta\rangle}(B),
$ while $A\ncong B$.
But there exists an ordered scaled graded monoid isomorphism $\eta:\,\underline{V}(A)\to \underline{V}(B),$
whose $\mathrm{Gr}(\eta):\, \underline{\rm K}(A)\to \underline{\rm K}(B)$ (via Proposition \ref{recover prop}, we identify $\mathrm{Gr}(\underline{\rm V}(A))$ with $\underline{\rm K}(A)$) is ordered, graded isomorphism, but not $\Lambda$-linear.
\end{example}

\begin{remark}
The example above shows that it is necessary to check the $\Lambda$ structure of the invariants between the algebras. We point out that it is also necessary to keep the $\Lambda$ structure of the invariants between the ideals. One may refer to \cite{ALbock} for such technical discussions on the relations between commutativity and Bockstein operations.

\end{remark}

\subsection{Application to extension}
\begin{notion}\rm
Given pre-ordered monoids  $S$ and $T$, and $\phi:\, S \to T$ a pre-order preserving monoid morphism. 
Denote by ${\rm Im}(\phi)$ and  ${\rm Ker}(\phi)$ the image and the kernel of $\phi$ (i.e., the elements mapped to 0), in the standard sense.
\end{notion}
\begin{theorem}\label{exact V}
Let $A$ be a C*-algebra of real rank zero, $I$ is an ideal of $A$. Then we have
$$
0\to V(I)\xrightarrow{V(\iota)} V(A)\xrightarrow{V(\pi)} V(A/I)\to 0
$$
is exact with respect to ${\rm Im}(\iota)={\rm Ker}(\pi)$.
\end{theorem}
\begin{proof}
Consider the exact sequence:
$$
0\to I\otimes \mathcal{K}\xrightarrow{\iota} A\otimes \mathcal{K}\xrightarrow{\pi} (A/I)\otimes \mathcal{K}\to 0.
$$
By Proposition \ref{V-injective}, $V(\iota)$ is injective. By Proposition \ref{lin inj}, we have
$V(\pi)$ is surjective. And it is obvious that ${\rm Im}(V(\iota))\subset {\rm Ker}(V(\pi))$,
now we show $ {\rm Ker}(V(\pi))\subset {\rm Im}(V(\iota))$.

For any projection $p\in A\otimes \mathcal{K}$ with $[\pi(p)]=0\in V(A/I)$,
we have $\pi(p)\sim_{\rm Mv}0$, i.e., $[\pi(p)]_{{\rm Mv}}=0$. From exactness, we have $p\in {\rm Im}(\iota)$, i.e., $[p]_{\rm Mv}\in {\rm Im}(V(\iota))$.
\end{proof}
Without the assumption of real rank zero, the above result is not true. One may concern: $0\to C_0(0,1)\to C[0,1]\to \mathbb{C}\oplus \mathbb{C}\to 0.$

\begin{lemma}\label{infinite or sr1}
Given two extensions:
$$0\to B_i\xrightarrow{\iota_i} E_i \xrightarrow{\pi_i} A_i\to 0,\quad i=1,2,$$
with all $B_i, E_i, A_i$ are stable C*-algebras of real rank zero,
where $B_i$ is identified as its image---the ideal of $E_i$ through $\iota_i$.
Suppose we have  $\alpha:\, V(E_1)\to V(E_2)$ is a  pre-ordered monoid isomorphism and  $\alpha_0:\,V(B_1)\to V(B_2)$ is also a  pre-ordered monoid isomorphism as the restriction map of $\alpha$, i.e., we have
the following commutative diagram
$$
\xymatrixcolsep{2pc}
\xymatrix{
{\,\,0\,\,} \ar[r]^-{}
& {\,\,{V}(B_1)\,\,} \ar[d]_-{\alpha_0} \ar[r]^-{{V}(\iota_1)}
& {\,\,{V}(E_1)\,\,} \ar[d]_-{\alpha} \ar[r]^-{{V}( \pi_1) }
& {\,\,{V}(A_1)\,\,}  \ar[r]^-{}
& {\,\,0\,\,\,} \\
{\,\,0\,\,} \ar[r]^-{}
& {\,\,{V}(B_2)\,\,} \ar[r]_-{ {V}(\iota_2)}
& {\,\,{V}(E_2) \,\,} \ar[r]_-{{V}(\pi_2)}
& {\,\,{V}(A_2) \,\,} \ar[r]_-{}
& {\,\,0\,\,}.}
$$
Then $A_1$ is infinite if and only if $A_2$ is infinite.
\end{lemma}
\begin{proof}
Suppose $A_1$ is infinite, i.e., $A_1$ has an infinite projection $p$ with $p=p'+e$, $p'\neq 0$, $e\neq 0$, $p' e=0$ and $p'\sim_{\rm Mv}p$.
Since $E_1$ has real rank zero, by Proposition \ref{lin inj}, we can lift $p'$ and $e$ to projections $p_0'$ and $e_0$ in $E_1$.
Particularly, $p_0', e_0\notin B_1$, i.e., $[p_0']_{\rm Mv}, [e_0]_{\rm Mv}\notin V(B_1)$.
As $E_1$ is stable,
let $s_1,s_2$ be two isometries in ${\mathcal M}(E_1)$ with $s_1s_1^*+s_2s_2^*=1$. (Such a pair $s_1, s_2$ are called $\mathcal{O}_2$-isometries.)
Denote  $p_1'=s_1 p_0' s_1^*$ and $e_1=s_2 e_0 s_2^*$, both of which are still projections in $E_1$
and  $$p_1 e_1=0,\quad p_1'\sim_{\rm Mv}p_0',\quad e_1\sim_{\rm Mv} e_0.$$
Now denote $p_1=p_1'+e_1$, then we have
$$\pi_1(p_1)=\pi_1(p_1')+\pi_1(e_1)\sim_{\rm Mv} p'+e=p\sim_{\rm Mv}p'\sim_{\rm Mv}\pi_1(p_1')\in A_1.
$$
This implies
$$\mathrm{K}_0(\pi_1)([p_1]_{\mathrm{K}_0})=\mathrm{K}_0(\pi_1)([p_1']_{\mathrm{K}_0}+[e_1]_{\mathrm{K}_0})
$$
and
$$
\mathrm{K}_0(\pi_1)([p_1]_{\mathrm{K}_0})=\mathrm{K}_0(\pi_1)([p_1']_{\mathrm{K}_0}).
$$
By  the six-term exact sequence, there exist projections $d,f,g,h$ in $B_1$
with
$$
[p_1]_{\mathrm{K}_0}=[p_1']_{\mathrm{K}_0}+[e_1]_{\mathrm{K}_0}+[d]_{\mathrm{K}_0} -[f]_{\mathrm{K}_0}
$$
and
$$
[p_1]_{{\rm K}_0}=[p_1']_{{\rm K}_0}+[g]_{\mathrm{K}_0} -[h]_{\mathrm{K}_0}.
$$
This means that there exist projections $s,t \in E_1$ with
$$
[p_1]_{\rm Mv}+[f]_{\rm Mv}+[s]_{\rm Mv}=[p_1']_{\rm Mv}+[e_1]_{\rm Mv}+[d]_{\rm Mv} +[s]_{\rm Mv}
$$
and
$$
[p_1]_{\rm Mv}+[h]_{\rm Mv}+[t]_{\rm Mv}=[p_1']_{\rm Mv}+[g]_{\rm Mv}+[t]_{\rm Mv}.
$$
Denote $\kappa=[s]_{\rm Mv}+[t]_{\rm Mv}$, we have
$$
(1):\quad[p_1]_{\rm Mv}+[f]_{\rm Mv}+\kappa=[p_1']_{\rm Mv}+[e_1]_{\rm Mv}+[d]_{\rm Mv} +\kappa
$$
and
$$
(2):\quad[p_1]_{\rm Mv}+[h]_{\rm Mv}+\kappa=[p_1']_{\rm Mv}+[g]_{\rm Mv}+\kappa.
$$
Since $E_2$ is stable, (by using the seven isometries in $\mathcal{O}_7\subset \mathcal{M}(E_2)$,) we can take mutually orthogonal projections $p_2', e_2,d_2,f_2,g_2,h_2,\xi$ in $E_2$ with
$$
[p_2']_{\rm Mv}=\alpha([p_1']_{\rm Mv}),\quad
[e_2]_{\rm Mv}=\alpha([e_1]_{\rm Mv}),
$$$$
[d_2]_{\rm Mv}=\alpha_0([d]_{\rm Mv}),\quad
[f_2]_{\rm Mv}=\alpha_0([f]_{\rm Mv}),
$$$$
[g_2]_{\rm Mv}=\alpha_0([g]_{\rm Mv}),\quad
[h_2]_{\rm Mv}=\alpha_0([h]_{\rm Mv}),\quad
[\xi]_{\rm Mv}=\alpha(\kappa).
$$
Set $p_2=p_2'+e_2$, then
$$
[p_2]_{\rm Mv}=\alpha([p_1']_{\rm Mv})+\alpha([e_1]_{\rm Mv})=\alpha([p_1]_{\rm Mv}).
$$
In $V(E_2)$, $(1)$ and $(2)$ induce
$$
(1'):\quad[p_2]_{{\rm Mv}}+[f_2]_{{\rm Mv}}+[\xi]_{{\rm Mv}}=[p_2']_{{\rm Mv}}+[e_2]_{MV}+[d_2]_{\rm Mv} +[\xi]_{\rm Mv}
$$
and
$$
(2'):\quad[p_2]_{\rm Mv}+[h_2]_{\rm Mv}+[\xi]_{\rm Mv}=[p_2']_{\rm Mv}+[g_2]_{\rm Mv}+[\xi]_{\rm Mv}.
$$
When we map the above into $V(A_2)$, we obtain
$$
(1''):\quad[\pi_2(p_2+\xi)]_{\rm Mv}=[\pi_2(p_2'+e_2+\xi)]_{\rm Mv}
$$
and
$$
(2''):\quad[\pi_2(p_2+\xi)]_{\rm Mv}=[\pi_2((p_2'+\xi)]_{\rm Mv}.
$$
Particularly, we have
$$(1'''):\quad\pi_2(p_2+\xi)=\pi_2(p_2'+\xi)+\pi_2(e_2).
$$
Note that $p_2, e_2\in E_2\backslash B_2$, i.e., $\pi_2(p_2)$ and $\pi_2(e_2)$ are non-zero projections in $A_2$.  So $(1''')$ and $(2'')$ implies $\pi_2(p_2+\xi)$ is an infinite projection in $A_2$.
\end{proof}

\begin{remark}\rm\label{v total exact not}
Let $A$ be a separable C*-algebra  of real rank zero, $I$ is an ideal of $A$. However, unlike Theorem \ref{exact V}, we
don't always have
$$
0\to \underline{V}(I)\xrightarrow{\underline{V}(\iota)} \underline{V}(A)\xrightarrow{\underline{V}(\pi)} \underline{V}(A/I)\to 0
$$
is exact ($\underline{V}(\pi)$ may not be surjective),
though $\underline{V}(\iota)$ is always injective by Theorem \ref{total lat iso}.
If one assume $\underline{V}(\pi)$ is surjective, by Theorem \ref{ordertotal} (i) and Theorem \ref{Gr V total}, one will get
$\underline{\mathrm{K}}(\pi)$ surjective. We refer the readers to \cite[Remark 2.4]{AL} for an example whose $\underline{\mathrm{K}}(\pi)$ is not surjective.


\end{remark}











\section{Classification}

\subsection{Corona factorization property}

\begin{definition}\rm
A C*-algebra $A$ is said to be {\it purely infinite} if $A$ has no nonzero abelian quotients and if for every pair of positive elements $a, b$ in $A$, such that $a$ lies in the closed two-sided ideal generated by $b$, there exists a sequence $\left\{r_j\right\}_{j=1}^{\infty}$ in $A$ with $r_j^* b r_j \rightarrow a$.
\end{definition}
The following class is named to recognize the importance of Eberhard Kirchberg's contributions to the classification of these $C^*$-algebras.

\begin{definition}[\cite{Rbook}]\rm
  A {\it Kirchberg algebra} is a simple, purely infinite, nuclear, separable C*-algebra.
\end{definition}

The corona factorization property is an interesting algebraic condition to characterize  absorbing extensions; see Ng's survey paper \cite{Ng}.
\begin{definition}\rm
  Let $B$ be a separable stable C*-algebra. Then $B$ is said to have the {\it corona factorization
property}, if every norm-full projection in $\mathcal{M}(B)$ is Murray-von Neumann equivalent to $1_{\mathcal{M}(B)}$.
\end{definition}
\begin{proposition} {\rm (}\cite[Theorem 1.4]{KN2},\cite{Ng}{\rm)} \label{cfp abs}
If $B$ is a separable stable C*-algebra with the corona factorization property and $A$ is unital nuclear, then every unital norm-full extension of $A$ by $B$ is absorbing.
\end{proposition}

\begin{remark} \label{CFP ex}
Many classes of separable C*-algebras are known to have the corona factorization property, e.g., any Kirchberg algebra (\cite[Proposition 2.1]{Ng}), all C*-algebras with finite nuclear dimension (\cite[Corollary 3.5]{R CFP}). Note that, by \cite[Section 4]{KW} and \cite{WZ}, an A$\mathcal{HD}$ algebra (\ref{AHD}) has nuclear dimension no more than two.
\end{remark}



\subsection{Main classification theorem}

\begin{definition}\rm
Given two separable C*-algebras $A,B$, we say
$\underline{V}(A)$ and $\underline{V}(B)$ are {\it latticed isomorphic}, denoted by
$\underline{V}(A)\cong \underline{V}(B),$ if
there is a $\underline{\mathcal{V}}$-isomorphism $\phi:\,\underline{V}(A)\to \underline{V}(B)$.

We say
$\underline{V}(A)$ and $\underline{V}(B)$ are {\it latticed scaled isomorphic}, denoted by
$$(\underline{V}(A),\Delta(A))\cong (\underline{V}(B),\Delta(B)),$$ if
there is a $\underline{\mathcal{V}}_u$-isomorphism $\phi:\,(\underline{V}(A),\Delta(A))\to (\underline{V}(B),\Delta(B))$.
\end{definition}
\begin{theorem}\label{Gabe crelle}{\rm (}\cite[Corollary 6.14]{Gabe crelle}{\rm )}
Let $A$ and $B$ be separable, nuclear C*-algebras which are either both
stable or both unital. Then $A \otimes \mathcal{O}_2 \cong B \otimes \mathcal{O}_2$ if and only if ${\rm Lat}(A)\cong{\rm Lat}(B)$.
More concretely, each isomorphism $\alpha:\,{\rm Lat}(A) \to {\rm Lat}(B)$ can be lifted to
a $*$-isomorphism $\psi:\, A \otimes \mathcal{O}_2\to B \otimes \mathcal{O}_2$.
\end{theorem}
\begin{corollary}\label{gabe corollay}
Let $A$ and $B$ be separable, nuclear, $\mathcal{O}_2$-stable (i.e., $A\otimes \mathcal{O}_2\cong A$) C*-algebras of real rank zero, which are either
both stable or both unital. Then $A\cong B$ if and only if $\underline{V}_u(A)\cong \underline{V}_u(B)$.
\end{corollary}
\begin{proof}
Given a  $\underline{\mathcal{V}}_u$-isomorphism $\phi:\,(\underline{V}(A),\Delta(A))\to (\underline{V}(B),\Delta(B))$,
combining with Theorem \ref{total lat iso} and Proposition \ref{v(a) Lambda cong},
$\phi$  induces a lattice isomorphism:
$$\mathrm{Lat}(\phi):\, \mathrm{Lat}(A)\cong\mathrm{Lat}(\underline{V}(A))\to \mathrm{Lat}(\underline{V}(B))\cong\mathrm{Lat}(B)$$
in an obvious way.
Thus, by Theorem \ref{Gabe crelle}, we conclude the proof.
\end{proof}

\begin{proposition} \label{O2 fin nuc}
  Let $A$ be a separable nuclear $\mathcal{O}_2$-stable $C^*$-algebra. Then $A$ has nuclear dimension no more than $3$.
\end{proposition}
\begin{proof}
 Since every separable nuclear $\mathcal{O}_{\infty}$-stable $C^*$-algebra has nuclear dimension $1$ (\cite[Theorem A]{BGSW}) and $\mathcal{O}_{2}\cong\mathcal{O}_{2}\otimes\mathcal{O}_{\infty}$, this result follows directly from Proposition 2.3 and Theorem 7.4 in \cite{WZ}:
  $$
{\rm dim}_{\rm nuc}A={\rm dim}_{\rm nuc}(A\otimes\mathcal{O}_2)={\rm dim}_{\rm nuc}(A\otimes \mathcal{O}_{\infty} \otimes\mathcal{O}_2)\leq 3.
$$
\end{proof}
\begin{notion}\label{L class}\rm
Denote by $\mathcal{L}$ the collection of all the countably direct sums of algebras in the following classes:

($\mathcal{L}$1) all A$\mathcal{HD}$ algebras of real rank zero;

($\mathcal{L}$2) all Kirchberg algebras  satisfying UCT;

($\mathcal{L}$3) all separable nuclear non-simple minimal
$\mathcal{O}_2$-stable algebras of real rank zero (minimal means the algebra can not be written as a direct sum of two or more non-trivial  such algebras).

Note that by Theorem \ref{Gabe crelle}, unital simple $\mathcal{O}_2$-stable algebra is unique, hence, it is $\mathcal{O}_2$---a Kirchberg algebra, i.e., $\mathcal{O}_2$  is a block in  class ($\mathcal{L}$2), but not a block in class ($\mathcal{L}$3), though $\mathcal{O}_2$ is minimal.

Given any algebra $B$ in $\mathcal{L}$, $B$ is a  C*-algebra of real rank zero together with a unique decomposition (up to isomorphism) $B=B^1\oplus B^2\oplus B^3$,
where $B^1$ is an A$\mathcal{HD}$ algebra, $B^2$ is an infinite or finite direct sum of UCT Kirchberg algebras and
$B^3$ is an infinite or finite direct sum of non-simple minimal $\mathcal{O}_2$-stable C*-algebras.
(Note that though $\bigoplus_{\mathbb{N}}\mathcal{O}_2$ is a non-simple $\mathcal{O}_2$-stable C*-algebra,
it can be a part of $B^2$, but can't be a part of $B^3$. The reason is that we require each summand of $B^3$ is non-simple and minimal.) Since $B^2$ has finite nuclear dimension (see Proposition 2.3 and Theorem 7.5 in \cite{WZ}), then by Remark \ref{CFP ex} and Proposition \ref{O2 fin nuc}, $B$ has finite nuclear dimension, and hence,  $B\otimes \mathcal{K}$ has the corona factorization property.
\end{notion}

We are now set to establish the following main theorem.

\begin{theorem} \label{main thm}
Let $A_1,A_2, B_1,B_2\in \mathcal{L}$, assume $B_1,B_2$ are stable and $A_1,A_2$ are  unital simple.
Then for any two unital extensions with trivial boundary maps
$$
e_i:\quad0\rightarrow B_i\xrightarrow{\iota_i} E_i\xrightarrow{\pi_i}A_i \to 0,\quad i=1,2,
$$
Then $E_1\cong E_2$ if and only if  $\underline{V}_u(E_1)\cong \underline{V}_u(E_2)$.



\end{theorem}
\begin{proof}
Note that by assumption  and Proposition \ref{lin inj}, we have $E_1, E_2$ are separable algebras of real rank zero. But $E_1,E_2$
are not necessarily of stable rank one, as the algebras in $\mathcal{L}$  do not always have stable rank one,
though we assume both the extensions have trivial boundary maps.

By \cite[Proposition 3.4]{ERR}, $B_i$ is the largest proper ideal of $E_i$, $i=1,2$. Suppose we have $\alpha:\,\underline{V}_u(E_1)\cong \underline{V}_u(E_2)$, as
both $E_1,$ $E_2$ are unital,  we obtain a restriction $\alpha:\,(\underline{V}(E_1),[1_{E_1}]_{\rm Mv})\cong (\underline{V}(E_2),[1_{E_2}]_{\rm Mv})$.
Then by Theorem \ref{total lat iso}, for each $i=1,2$, $V(B_i)$ is the unique maximal ideal of $V(E_i)$.
Thus, we obtain $\alpha_0$ as the restriction map of $\alpha$.
That is, we have the following commutative diagram:
$$
\xymatrixcolsep{2pc}
\xymatrix{
 {\,\,\underline{V}(B_1)\,\,} \ar[d]_-{\alpha_0} \ar[r]^-{\underline{V}(\iota_1)}
& {\,\,(\underline{V}(E_1),[1_{E_1}]_{\rm Mv})\,\,} \ar[d]_-{\alpha} \ar[r]^-{\underline{V}( \pi_1) }
& {\,\,(\underline{V}(A_1),[1_{A_1}]_{\rm Mv})\,\,} \\
{\,\,\underline{V}(B_2)\,\,} \ar[r]_-{\underline{V}(\iota_2)}
& {\,\,(\underline{V}(E_2),[1_{E_2}]_{\rm Mv}) \,\,} \ar[r]_-{\underline{V}(\pi_2)}
& {\,\,(\underline{V}(A_2),[1_{A_2}]_{\rm Mv}) \,\,} .}
$$
By Lemma \ref{infinite or sr1} and the definition of $\mathcal{L}$, we have $A_1,A_2$ are both simple A$\mathcal{HD}$
algebras of real rank zero  or both simple Kirchberg algebras.

From \ref{L class}, for each $i=1,2$, we write $V(B_i)=\underline{V}(B_i^1)\oplus \underline{V}(B_i^2)\oplus \underline{V}(B_i^3) $,
where
$B_i^1$ is an A$\mathcal{HD}$ algebra of real rank zero, $B_i^2$ is a  direct sum of  Kirchberg algebras,
$B_i^3$ is a  direct sum of non-simple minimal $\mathcal{O}_2$-stable algebras of real rank zero. By Proposition \ref{cancel or infinite} and Proposition \ref{total lat iso},
$\alpha_0$ can be  decomposed into $\alpha_0^1\oplus \alpha_0^2\oplus \alpha_0^3$, where $\alpha_0^k:\,\underline{V}(B_1^k)\to \underline{V}(B_2^k)$ is a $\underline{\mathcal{V}}_u$-isomorphism for each $k=1,2,3$.

By Proposition \ref{recover prop},
we obtain the following commutative diagram
$$\xymatrixcolsep{2pc}
\xymatrix{
 {\,\,\underline{\mathrm{K}}(B_1)\,\,} \ar[d]_-{H(\alpha_0)}\ar[r]^-{\underline{\mathrm{K}}(\iota_1)}
& {\,\,(\underline{\mathrm{K}}(E_1),[1_{E_1}]_{\mathrm{K}_0})\,\,} \ar[d]_-{H(\alpha)} \ar[r]^-{\underline{\mathrm{K}}(\pi_1)}
& {\,\,(\underline{\mathrm{K}}(A_1),[1_{A_1}]_{\mathrm{K}_0})\,\,}\\
{\,\,\underline{\mathrm{K}}(B_2)\,\,} \ar[r]_-{\underline{\mathrm{K}}(\iota_2)}
& {\,\,(\underline{\mathrm{K}}(E_2),[1_{E_2}]_{\mathrm{K}_0}) \,\,} \ar[r]_-{\underline{\mathrm{K}}(\pi_2)}
& {\,\,(\underline{\mathrm{K}}(A_2),[1_{A_2}]_{\mathrm{K}_0}) \,\,} .}
$$
Particularly, for each $k=1,2,3$, we have
$$H(\alpha_0^k):\,(\underline{\mathrm{K}}(B_1^k),\underline{\mathrm{K}}(B_1^k)_+,{\mathrm{K}}_0^+(B_1^k))\to (\underline{\mathrm{K}}(B_2^k),\underline{\mathrm{K}}(B_2^k)_+,{\mathrm{K}}_0^+(B_2^k))$$
is an ordered scaled $\Lambda$-isomorphism.
By \cite[Theorem 9.1]{DG}, we can lift $H(\alpha_0^1)$ to an isomorphism $\phi_0^1:\,B_1^1\to B_2^1$.
By Kirchberg-Phillips Classification Theorem (see \cite[Theorem 8.4.1]{Rbook}), we can lift $(H(\alpha_0^2))^*$ (the restriction of $H(\alpha_0^2)$ on $\mathrm{K}_*$)
$$(H(\alpha_0^2))^*:\,({\mathrm{K}}_*(B_1^2),{\mathrm{K}}_*^+(B_1^2),{\mathrm{K}}_0^+(B_1^2))\to ({\mathrm{K}}_*(B_2^2),{\mathrm{K}}_*^+(B_2^2),{\mathrm{K}}_0^+(B_2^2))$$
to an isomorphism $\phi_0^2:\,B_1^2\to B_2^2$.
By the proof of Corollary \ref{gabe corollay}, we lift
$${\mathrm{Lat}(\alpha_0^3)}:\, {\rm Lat}(B_1^3)\to {\rm Lat}(B_2^3)$$ to an isomorphism $\phi_0^3:\,B_1^3\to B_2^3$.
(Recall that $\mathcal{O}_2$-stable C*-algebras have trivial $ \underline{\mathrm{K}}$-groups, and hence,
$\phi_0^3$ induces $H(\alpha_0^3)$ as the zero map.)

Denote $\phi_0=\phi_0^1\oplus \phi_0^2\oplus \phi_0^3$.
Restricting on $\mathrm{K}_*$, we have the following commutative diagram with exact rows
$$
\xymatrixcolsep{2pc}
\xymatrix{
{\,\,0\,\,} \ar[r]^-{}
& {\,\,\mathrm{K}_*(B_1)\,\,} \ar[d]_-{\mathrm{K}_*(\phi_0)} \ar[r]^-{\mathrm{K}_*(\iota_1)}
& {\,\,(\mathrm{K}_*(E_1),[1_{E_1}]_{\mathrm{K}_0})\,\,} \ar[d]_-{{({H}(\rho)})^*} \ar[r]^-{\mathrm{K}_*(\pi_1)}
& {\,\,(\mathrm{K}_*(A_1),[1_{A_1}]_{\mathrm{K}_0})\,\,} \ar[d]_-{\varrho} \ar[r]^-{}
& {\,\,0\,\,} \\
{\,\,0\,\,} \ar[r]^-{}
& {\,\,\mathrm{K}_*(B_2)\,\,} \ar[r]_-{  \mathrm{K}_*(\iota_2)}
& {\,\,(\mathrm{K}_*(E_2),[1_{E_2}]_{\mathrm{K}_0}) \,\,} \ar[r]_-{\mathrm{K}_*(\pi_2)}
& {\,\,(\mathrm{K}_*(A_2),[1_{A_2}]_{\mathrm{K}_0}) \,\,} \ar[r]_-{}
& {\,\,0\,\,},}
$$ 
where $\mathrm{K}_*(\phi_0)$ is induced by $\phi_0$,  ${({H}(\rho)})^*$ is the restriction map of ${H}(\rho)$ between $\mathrm{K}_*$-groups
and  $\varrho$ is the induced map obtained from $\mathrm{K}_*(\phi_0)$ and $({{H}(\rho)})^*$.

Note that both $\mathrm{K}_*(\phi_0)$ and ${({H}(\rho)})^*$ are scaled  order-preserving maps, Proposition \ref{lin inj} implies that
 $\varrho$ is also a scaled  order-preserving map.

Recall that we have shown that $A_1,A_2$ are both simple A$\mathcal{HD}$ algebras of real rank zero  or both
simple Kirchberg algebras.
By \cite[Theorem 9.4]{DG} or Kirchberg-Phillips Classification Theorem (see \cite[Theorem 8.4.1]{Rbook}),
we can lift $\varrho$ to a unital isomorphism $\phi_1:\,A_1\to A_2$.
Now we have the following commutative diagram with exact rows:
$$
\xymatrixcolsep{1.8pc}
\xymatrix{
{\,\,0\,\,} \ar[r]^-{}
& {\,\,\mathrm{K}_*(B_1)\,\,} \ar[d]_-{{\rm id}} \ar[r]^-{\mathrm{K}_*(\iota_1)}
& {\,\,(\mathrm{K}_*(E_1),[1_{E_1}]_{\mathrm{K}_0})\,\,} \ar[d]_-{{\underline{H}(\rho)}^*} \ar[r]^-{\mathrm{K}_*(\phi_1\circ \pi_1) }
& {\,\,(\mathrm{K}_*(A_2),[1_{A_2}]_{\mathrm{K}_0})\,\,} \ar[d]_-{{\rm id}} \ar[r]^-{}
& {\,\,0\,\,\,} \\
{\,\,0\,\,} \ar[r]^-{}
& {\,\,\mathrm{K}_*(B_1)\,\,} \ar[r]_-{ \mathrm{K}_*( \iota_2\circ\phi_0)}
& {\,\,(\mathrm{K}_*(E_2),[1_{E_2}]_{\mathrm{K}_0}) \,\,} \ar[r]_-{\mathrm{K}_*(\pi_2)}
& {\,\,(\mathrm{K}_*(A_2),[1_{A_2}]_{\mathrm{K}_0}) \,\,} \ar[r]_-{}
& {\,\,0\,\,}.}
$$

That is,  by Theorem \ref{strong wei}, the following two unital extensions $f_1,f_2$ with trivial boundary maps
$$
f_1:\quad 0\to B_1 \xrightarrow{\iota_1} E_1\xrightarrow{\phi_1\circ \pi_1} A_2\to 0
$$
and
$$
f_2:\quad0\to B_1 \xrightarrow{\iota_2\circ \phi_0} E_2\xrightarrow{\pi_2} A_2\to 0
$$
share the same equivalent class in $\mathrm{Ext}_{ss}^u(A_2, B_1)$.

By Definition \ref{def ext}, we have both $f_1$ and  $f_2$ are unital essential full extensions,
and hence, by Proposition \ref{cfp abs} and  Remark \ref{CFP ex}, 
both $f_1$ and  $f_2$ are absorbing.
Thus,
$f_1$ and  $f_2$ are strongly unitarily equivalent,  then by Definition \ref{strong tui cong}, we have  $E_1\cong E_2$.
\end{proof}
\begin{remark}
Consider the following unital extension
$$
e_1:\quad 0\to \mathcal{O}_2\otimes\mathcal{K}\to E_1\to \mathcal{O}_2\to 0.
$$
As $\mathcal{O}_2$ has trivial K-groups, by Proposition \ref{lin inj},
we have $E_1$ has real rank zero and trivial K-groups. Since finite nuclear is preserved under taking extensions and the nuclear dimension of $\mathcal{O}_n$ $(n=2,3,\cdots)$ is one \cite{WZ},
by Remark \ref{CFP ex}, $\mathcal{O}_2\otimes\mathcal{K}$ and $E_1$ have  the corona factorization
property, then $e_1$ is absorbing.
By Theorem \ref{strong wei}, $E_1$ is unique up to isomorphism.
It is not hard to see that $E_1\cong \widetilde{\mathcal{K}}\otimes \mathcal{O}_2,$ and hence, $E_1$ is $\mathcal{O}_2$-stable and
$$
\mathrm{Lat}(E_1)=\{0\subset \mathcal{O}_2\otimes\mathcal{K}\subset E_1\}.
$$

Inductively, for $n\geq 2$, one may construct the unital extension
$$
e_n:\quad 0\to E_{n-1}\otimes\mathcal{K}\to E_n\to \mathcal{O}_2\to 0,
$$
with $E_n$ is unique up to isomorphism and has real rank zero and trivial K-groups. Also $E_n\otimes\mathcal{K}$ has the corona factorization
property.
Moreover, all such $E_n$ are separable nuclear and  $\mathcal{O}_2$-stable.
Particularly,
$$
\mathrm{Lat}(E_n)=\{0\subset \mathcal{O}_2\otimes\mathcal{K}\subset E_1\otimes\mathcal{K}\subset \cdots\subset  E_{n-1}\otimes\mathcal{K}\subset E_n\}.
$$

One can even build the unital extension
$$\textstyle
e_n:\quad 0\to \bigoplus\limits_{i=1}^nE_{i}\otimes\mathcal{K}\to C_n\to \mathcal{O}_2\to 0,
$$
with that $C_n$ are separable nuclear and  $\mathcal{O}_2$-stable C*-algebras of real rank zero and $C_n\otimes \mathcal{K}$ has the corona factorization property.

In general, there are a large amount of algebras in the class  $(\mathcal{L}3)$.
\end{remark}

\begin{remark}
In theorem \ref{main thm}, the real rank zero algebras we classified are either infinite or of stable rank one. But the new invariant would even distinguish more algebras in the setting of real rank zero. For example, take the stably finite algebra $A_1$ of real rank zero and stable rank two constructed in \cite[2.8]{Go}, which comes from an extension of a Bunce-Deddens algebra by $\mathcal{K}\oplus\mathcal{K}$. $(\mathrm{K}_0(A_1),\mathrm{K}_0^+(A_1))$ is as same as an ordered $\mathrm{K}_0$-group of AF-algebra $A'$, while $\mathrm{K}_1(A_1)=0$. This implies the ordered $\mathrm{K}_*$-group and even the $\underline{\mathrm{K}}$ do not distinguish $A_1$ and $A'$. However, we do have$$V(A_1)\ncong V(A')\quad {\rm and}\quad \underline{V}(A_1)\ncong \underline{V}(A'),$$as $V(A_1)$ does not have cancelation and $V(A')$ does.

\end{remark}






\begin{thebibliography}{}


\bibitem{AELL}
Q. An, G. A. Elliott, Z. Li and Z. Liu. The classification of certain ASH C*-algebras of real rank zero. J. Topol. Anal., 14 (1) (2022), 183--202.

\bibitem{AL jfa}
Q. An and Z. Liu. A total Cuntz semigroup for C*-algebras of stable rank one. J. Funct. Anal., 284 (8) (2023), No. 109858.

\bibitem{AL}
Q. An and Z. Liu. Total Cuntz semigroup, extension, and Elliott conjecture with real rank zero. Proc. Lond. Math. Soc. (3), 128 (4) (2024), Paper No. e12595, 40 pp.

\bibitem{ALbock}
Q. An and Z. Liu. Bockstein operations and extensions with trivial boundary maps. Preprint.

\bibitem{ALZ}
Q. An, Z. Liu and Y. Zhang.  On the classification of certain real rank zero C*-algebras. Sci. China Math., 65 (4) (2022), 753--792.








\bibitem{L2}
L. Cantier. Unitary Cuntz semigroups of ideals and quotients. M\"{u}nster J. of Math., 14 (2) (2021), 585--606.


\bibitem{BGSW}
 J. Bosa, J. Gabe, A. Sims and S. White. The nuclear dimension of $\mathcal{O}_{\infty}$-stable C*-algebras. Adv. Math., 401 (2022), No.108250.







\bibitem{Cu alg}
J. Cuntz. Simple C*-algebras generated by isometries. Commun. Math. Phys., 57 (2) (1977), 173--185.




\bibitem{Cu1}
J. Cuntz. A new look at KK-theory. K-Theory,  1 (1) (1987), 31--51.


\bibitem {DL0}
M. Dadarlat and T. A. Loring. Classifying C*-algebras via ordered mod-p K-theory. Math. Ann., 305 (1996), 601--616.



\bibitem {DEb}
M. Dadarlat and S. Eilers, The Bockstein Map is Necessary. Canadian Mathematical Bulletin 42.3 (1997), 274--284.



\bibitem {DE}
M. Dadarlat and S. Eilers. Compressing coefficients while preserving ideals in the K-theory for C*-algebras. K-Theory, 14 (1998), 281--304.



\bibitem {DG}
M. Dadarlat and G. Gong. A classification result for approximately homogeneous C*-algebras of real rank zero. Geom. Funct. Anal., 7 (4) (1997), 646--711.



\bibitem {DL3}
M. Dadarlat and T. A. Loring. Classifying C*-algebras via ordered mod-p K-theory. Math. Ann., 305 (4) (1996), 601--616.


\bibitem {Ei}
S. Eilers. A complete invariant for AD algebras with bounded torsion in $\mathrm{K}_1$. J. Funct. Anal., 139 (2) (1996), 325--348.


\bibitem{EGKRT} S. Eilers, J. Gabe, T. Katsura, E. Ruiz and M. Tomforde.  The extension problem for graph C*-algebras. Annals of K-theory, 5 (2) (2020), 295--315.

\bibitem{ERZ}
S. Eilers, G. Restorff and E. Ruiz. Classification of extensions of classifiable
C*-algebras. Adv. Math., 222 (6) (2009), 2153--2172.

\bibitem{ERR}
S. Eilers, G. Restorff and  E. Ruiz. The ordered K-theory of a full extension. Canad. J. Math., 66 (3) (2014), 596--625.





\bibitem {Ell1}
G. A. Elliott. On the classification of inductive limits of sequence of semisimple finite-dimensional algebras, J. Algebra, 38 (1976), 29--44.

\bibitem {Ell}
G. A. Elliott. On the classification of C*-algebras of real rank zero. J.
Reine Angew. Math., 443 (1993), 179--219.



\bibitem{EGLN}
G. A. Elliott, G. Gong, H. Lin and Z. Niu. On the classification of simple
amenable $\mathrm{C}^*$-algebras with finite decomposition rank. II, to appear on J. Noncommut. Geom.

\bibitem{EGLN2}
G. A. Elliott, G. Gong, H. Lin and Z. Niu. The classification of simple separable unital locally ASH algebras. J. Funct. Anal., 272 (12) (2017), 5307--5359.

\bibitem{EGS}
G. A. Elliott, G. Gong and H. Su. On the classification of C*-algebras of real rank zero. IV. Reduction to local spectrum of dimension two. Operator algebras and their applications, II (Waterloo, ON, 1994/1995), 73--95.
Fields Inst. Commun., 20.
American Mathematical Society, Providence, RI, 1998.





\bibitem{Gabe crelle}
J. Gabe. A new proof of
Kirchberg's $\mathcal{O}_2$-stable classification.
J.
Reine Angew. Math., 761 (2020),  247--289.

\bibitem{GaLN}
J. Gabe, H. Lin and P. W. Ng. Extensions of C*-algebras. arxiv.org/abs/2307.15558v1

\bibitem{GR}
J. Gabe and E. Ruiz. The unital Ext-groups and classification of C*-algebras. Glasgow Math. J., 62 (2020), 201--231.


\bibitem{G}
G. Gong. Classification of C*-Algebras of Real Rank Zero and Unsuspended E-Equivalence Types. J. Funct. Anal., 152 (2) (1998), 281--329.


\bibitem{GJL}
G. Gong, C. Jiang and L. Li. Hausdorffified algebraic $\mathrm{K}_1$-group and invariants for C*-algebras with the ideal property. Ann. K-Theory, 5 (1) (2020), 43--78.

\bibitem{GJL2}
G. Gong, C. Jiang and L. Li. A classification of inductive limit C*-algebras with ideal property. Trans. London Math. Soc., 9 (1) (2022), 158--236.

\bibitem{GJLP1}
G. Gong, C. Jiang, L. Li and C. Pasnicu. A$\mathbb{T}$  structure of  AH  algebras with the ideal property and torsion free  K-theory.
J. Funct. Anal., 258 (6) (2010), 2119--2143.

\bibitem{GJLP2}
G. Gong, C. Jiang, L. Li and C. Pasnicu.  A reduction theorem for  AH  algebras with the ideal property.
Int. Math. Res. Not. IMRN, (24) (2018), 7606--7641.


\bibitem{GLN1}
G. Gong, H. Lin and Z. Niu. A classification of finite simple amenable $\mathcal{Z}$-stable
C*-algebras, I: C*-algebras with generalized tracial rank one. C. R. Math. Acad. Sci. Soc. R. Can., 42 (3) (2020), no. 3, 63--450.

\bibitem{GLN2}
G. Gong, H. Lin and Z. Niu. A classification of finite simple amenable $\mathcal{Z}$-stable
C*-algebras, II: C*-algebras with rational generalized tracial rank one. C. R. Math. Acad. Sci. Soc. R. Can., 42 (4) (2020), 451--539.


\bibitem{Go}
K. R. Goodearl. C*-algebras of real rank zero whose $\mathrm{K}_0$'s are not Riesz groups. Canad. Math. Bull., 39 (4) (1996), 429--437.




\bibitem{KN2}
D. Kucerovsky and P. W. Ng. $S$-regularity and the corona factorization property.
Math. Scand., 99 (2) (2006), 204--216.



\bibitem{KW}
E. Kirchberg and  W. Winter.
Covering dimension and quasidiagonality.
Internat. J. Math., 15 (1) (2004), 63--85.

\bibitem{L0}
H. Lin. ${\rm C}^*$-algebra extensions of  $C(X)$. Mem. Amer. Math. Soc., 115 (550) (1995), vi+89 pp.

\bibitem{L1}
H. Lin. Extensions by ${\rm C}^*$-algebras with real rank zero. II, Proc. London Math. Soc., 71 (3) (1995), 641–-674.

\bibitem{L00}
H. Lin. Extensions by ${\rm C}^*$-algebras of real rank zero. III, Proc. London Math. Soc., 76 (3) (1998), 634--666.

\bibitem{L}
H. Lin. An Introduction to the Classification of Amenable $\mathrm{C}^*$-Algebras. World Scientific Publishing Co. Inc., River Edge, NJ, 2001.




\bibitem{LN}
H. Lin and P. W. Ng.
Extensions of $\mathrm{C}^*$-algebras by a small ideal. Int. Math. Res. Not. IMRN, 12 (2023), 10350--10438.


\bibitem{LR}
H. Lin and M. R\o rdam. Extensions of inductive limits of circle algebras. J. London Math. Soc., 51 (2) (1995), 603--613.

\bibitem{LS}
H. Lin and H. Su. Classification of direct limits of generalized Toeplitz algebras. Pacific J. Math., 181 (1) (1997), 89--140.





\bibitem{Ng}
P. W. Ng. The corona factorization property. Operator theory, operator algebras, and applications, 97--110, Contemp. Math., 414, Amer. Math. Soc., Providence, RI, 2006.


\bibitem{R CFP}
L. Robert. Nuclear dimension and $n$-comparison. M\"{u}nster J. of Math., 4 (2011), 65--72



\bibitem{RS}
J. Rosenberg and C. Schochet. The K\"{u}nneth theorem and the universal coefficient theorem for Kasparov's generalized $\mathrm{K}$-functor. Duke Math. J., 55 (1987), 431--474.


\bibitem{Rbook}
M. R\o rdam.  Classification of Nuclear C*-algebras. Springer Berlin Heidelberg. (2002).

\bibitem{R2}
 M. R\o rdam. Classification of extensions of certain $C^*$-algebras by their six term exact sequences in K-theory. Math. Ann., 308 (1997), 93--117.

\bibitem{R acta}
M. R\o rdam, A simple C*-algebra with a finite and an infinite projection.  Acta Math., 191 (1) (2003): 109--142.





\bibitem{Sk1}
G. Skandalis. On the strong Ext bifunctor. Preprint. (1983)

\bibitem{Sk2}
G. Skandalis. Kasparovs bivariant K-theory and applications. Expos. Math., 9 (1991), 193–-250.



\bibitem{TWW}
A. Tikuisis, S. White and W. Winter. Quasidiagonality of nuclear {$C^\ast$}-algebras. Ann. of Math. (2), 185 (1) (2017),  229--284.


\bibitem{V}
J. Villadsen. On the stable rank of simple C*-algebras. J. Amer. Math. Soc., 12 (4) (1999),  1091--1102.

\bibitem{W}
C. Wei.  On the classification of certain unital extensions of
C*-algebras. Houston J.  Math., 41 (3) (2015), 965--991.




\bibitem{WZ}
W. Winter and J. Zacharias.
The nuclear dimension of C*-algebras. Adv. Math., 224 (2) (2010), 461--498.



\bibitem{Z}
S. Zhang. A Riesz decomposition property and ideal structure of multiplier algebras. J. Operator Theory, 24 (1990), 209--225.



\end{thebibliography}
\end{document}